\documentclass[a4paper,11pt]{amsart}
\usepackage[utf8]{inputenc}
\usepackage[a4paper]{geometry} 
\usepackage{mathtools,amssymb,bm,amsthm,amsmath} 
\usepackage{amsfonts}
\usepackage[utf8]{inputenc}
\usepackage[english]{babel} 
\usepackage{xcolor} 
\usepackage{graphicx} 
\usepackage[pdfencoding=auto]{hyperref} 
\usepackage{stmaryrd, enumitem}
\usepackage{dsfont}

\newtheorem{theorem}{Theorem}[section]
\newtheorem{proposition}[theorem]{Proposition}

\newtheorem{lemma}[theorem]{Lemma}
\newtheorem{fact}[theorem]{Fact}
\newtheorem{claim}[theorem]{Claim}

\theoremstyle{remark}
\newtheorem*{remark}{Remark}

\theoremstyle{definition}
\newtheorem{definition}{Definition}[section]

\numberwithin{equation}{section}
\newcommand{\R}{\mathbb{R}}
\newcommand{\N}{\mathbb{N}}

\newcommand{\C}{\mathbb{C}}

\newcommand{\F}{\mathcal{F}}
\newcommand{\Sc}{\mathcal{S}}
\newcommand{\Sp}{\mathbb{S}}

\newcommand{\minus}{\setminus}
\newcommand{\vertiii}[1]{{\left\vert\kern-0.25ex\left\vert\kern-0.25ex\left\vert #1 
		\right\vert\kern-0.25ex\right\vert\kern-0.25ex\right\vert}}

\makeatletter
\@namedef{subjclassname@2020}{%
	\textup{2020} Mathematics Subject Classification}
\makeatother

\author[M. Fakhoury]{Micheline Fakhoury}
\address{Univ. Artois, UR 2462, Laboratoire de Math\'{e}matiques de Lens (LML)\\ F-62300 Lens, France}
\email{micheline.fakhoury@univ-artois.fr}

\begin{document}
	\emergencystretch 3em
	
	\title[Tingley's problem for Schreier spaces]{Tingley's problem for Schreier spaces \\and their  $p\,$-$\,$convexifications}

	\begin{abstract} We describe the surjective isometries of the unit sphere of real  Schreier spaces of all orders and their $p\,$-$\,$convexifications, for $1<p<\infty$. This description allows us to provide for those spaces a positive answer  to a special case of Tingley's problem, which asks whether every surjective isometry of the unit sphere of a real Banach space can be extended to a linear isometry of the entire space.
	\end{abstract}
	
	\keywords{Tingley's problem, isometries, unit sphere, combinatorial spaces, Schreier spaces, $p\,$-$\,$convexifications.}
	\subjclass[2020]{46B04, 46B45, 03E05}
	%\thanks{The author would like to thank her advisor \'Etienne Matheron for his help in the writing of the paper, and the referee for his/her valuable suggestions.}
	\thanks{This work was supported in part by
		the project FRONT of the French
		National Research Agency (grant ANR-17-CE40-0021)}
	
	\maketitle
	
	\section{Introduction}
	The so-called \emph{Tingley's problem} has been intensively studied in recent years. It has the following simple formulation,  originally posed by Tingley in~\cite{T}: Let $X$ and $Y$ be two real Banach spaces and let $S_X$ and $S_Y$ be their unit spheres, respectively. Does every surjective isometry $T: S_X \to S_Y$ extend to a linear  isometry $\widetilde  T : X\to Y$?
	
	Note that this problem also makes sense for complex Banach spaces; however, the best one can hope for is that $T$ can be extended to an  $\R$-linear isometry. For instance,  if $X=Y= \mathbb{C}$ and $T(x)=\Bar{x}$ for all $x\in S_X$, then $T$ cannot be extended to a $\C$-linear isometry of $X$.
	
	This intriguing problem has been studied for various Banach spaces, and several   partial positive answers have been obtained; see e.g. the surveys~\cite{Ding7, PSurvey, YZ}. It should be noted that this problem remains open for Banach spaces of dimension three or higher. In a notable advancement, T. Banakh~\cite{Banakh} provided a positive answer for \emph{two-dimensional} real Banach spaces. Moreover, the problem remains unsolved for the case where $X=Y$. Recently, N. Ma\'slany~\cite{Ma} obtained a positive answer  to this particular problem (when $X=Y$) for some real combinatorial Tsirelson spaces by describing the surjective isometries of their unit spheres. This motivated us to study the same problem for other combinatorial-like Banach spaces, in particular for  Schreier spaces and their  $p\,$-$\,$convexifications.
	
	Throughout the paper, all vector spaces are \emph{real}. We denote by $c_{00}$ the vector space of all finitely supported real-valued sequences, and by $(e_i)_{i\in\N}$ the canonical basis of $c_{00}$. 
	
	We point out that $\N$ starts at~$1$. By a \emph{combinatorial family} $\F$, we mean a family of finite subsets of $\N$ that contains all singletons and is hereditary for inclusion: if $F \in \F$ and $G \subset F$ then $G \in \F$. 
	
	Given a combinatorial family $\F$, the \emph{combinatorial Banach space} $X_\F$ is defined to be the completion of $c_{00}$  with respect to the following norm: 
	\[ \Vert x \Vert_\F := \sup \biggl\{\sum_{i\in F} \vert x(i)\vert; \; F \in \F\biggr\}. \]
	This is indeed a norm because $\mathcal F$ contains all singletons.
	
	Now, given a combinatorial family $\F$ and $p\in (1,\infty)$, the $p\,$-$\,$convexification of $X_\F$, denoted by $X_{\F,p}$, is defined to be the completion of $c_{00}$ with respect to the norm 
	\[ \Vert x \Vert_{\F,p} := \sup \biggl\{\biggl(\sum_{i\in F} \vert x(i)\vert^p\biggr)^{{1}/{p}}; \; F \in \F\biggr\}. \]
	Moreover, for convenience of notation, we consider that $X_{\F,1}=X_\F$. 
	
	It is clear by definition that $(e_i)_{i\in\N}$ is a $1$-unconditional Schauder basis for $X_{\F,p}$, $1\leqslant p<\infty $. 
	
	\medskip
	As mentioned above, in the present paper, we study  Tingley's problem for $X_{\F,p}$ where $\F$ is a generalized Schreier family $\Sc_\alpha$, defined below,  and $p\in [1,\infty)$. 
	
	\smallskip
	D. E. Alspach and S. A. Argyros introduced  in \cite{AA}  the so-called ``generalized Schreier families''. 
	Let $\omega_1$ be the first uncountable ordinal. We will say that a (transfinite) sequence 
	$(\Sc_\alpha)_{0\leqslant  \alpha<\omega_1}$ of families of finite subsets of $\N$ is a \emph{transfinite Schreier sequence} if
	
	\begin{itemize}
		\item $\Sc_0=\{ F\subset\N;\; \vert F\vert\leqslant 1\}$;
		\item $\Sc_{\beta +1}=\bigl\{\bigcup_{i=1}^{n}E_i; \; n \leqslant  E_1<E_2< \dots <E_n \text{ and }E_i\in \Sc_{\beta}\bigr\}$ for every $\beta<\omega_1$;
		\item $\Sc_{\alpha}= \{F \subset \N; \; F \in \Sc_{\alpha_n} \text{ for some }n \leqslant  F\}$ if $\alpha$ is a limit ordinal, where $(\alpha_n)_{n\in\N}$ is a preassigned strictly increasing sequence of non-zero ordinals such that $\sup_n\alpha_n=\alpha$. 
	\end{itemize}
	
	Note that in the definition of $\Sc_{\beta+1}$ and $\Sc_\alpha$ (when $\alpha$ is a limit ordinal), we assume  $n\leqslant  E$ is always true if $E=\emptyset$, and similarly  $E<E'$ is true  if either $E$ or $E'$ is $\emptyset$. 
	
	Also, it is important to mention that   there is  not a unique  transfinite Schreier sequence $(\Sc_\alpha)_{\alpha<\omega_1}$, because for limit ordinals $\alpha$, the family $\Sc_\alpha$ depends on the choice of the increasing sequences $(\alpha_n)$ converging to $\alpha$. On the other hand, $\Sc_k$ is uniquely determined for each natural number $k$; and in particular, $\Sc_k$ for $k:= 1$ is  the classical  \emph{Schreier family} \[ \mathcal S_1:=\{\emptyset\}\cup\{ \emptyset\neq F\subset\N;\; \vert F\vert\leqslant \min(F)\}.\]
	
	One can verify by transfinite induction that for any $ \alpha<\omega_1$, $\Sc_\alpha$ is a combinatorial family that is  spreading (for any $F \in \Sc_\alpha$ and $\sigma : F \to \N$ such that  $\sigma(n) \geqslant   n$ for all $n \in F$, 
	we have $\sigma(F) \in \Sc_\alpha$), and  compact when considered as a subset of $\{ 0, 1\}^\N$ \textit{via} the identification of a set $F\subset\N$ with its characteristic function, \textit{i.e.}  $\Sc_\alpha$ is closed in $\{ 0,1\}^\N$. Moreover, it can be verified easily by transfinite induction that for any $1\leqslant  \alpha<\omega_1$,   $\{m,n\}\in \Sc_\alpha$ for any $2\leqslant  m<n\in \N$, and $\Sc_\alpha^{MAX}=\{\{1\}\}$ where $\Sc_\alpha^{MAX}$ is the set of all maximal elements of $\Sc_\alpha$ with respect to inclusion.
	
	\smallskip
	Note that  $X_{\Sc_1}$ is called the \emph{Schreier space}, and given a  countable ordinal $\alpha<\omega_1$,  $X_{\Sc_\alpha}$ is called the \emph{Schreier space of order $\alpha$}. 
	
	\medskip
	Before stating our main theorem, we need to point out a property of transfinite Schreier sequences $(\Sc_\alpha)_{\alpha<\omega_1}$ that depends on the choice of the approximating sequences. 
	
	\begin{definition}\label{def}
		A transfinite Schreier sequence $(\Sc_\alpha)_{\alpha<\omega_1}$  will be  said to be \emph{good} if it has the following property: for every limit ordinal $\alpha<\omega_1$, the approximating sequence $(\alpha_n)$ is made up of  successor ordinals, $\alpha_n=\beta_n+1$, and the sequence 
		$(\Sc_{\beta_n})_{n\in\N}$ is increasing.
	\end{definition}
	Note that there exist transfinite Schreier sequences that are good by \cite[Proposition~3.1 (viii)]{Cau} (see also \cite[Fact 3.1]{
		F}).
	
	\medskip
	
	We are going to prove the following theorem.
	
	\begin{theorem}\label{main} Let $(\Sc_\alpha)_{\alpha<\omega_1}$ be a good transfinite Schreier sequence and let $p\in [1,\infty)$ and $1\leqslant \alpha<\omega_1$. Then every surjective isometry of the unit sphere of $X_{\Sc_\alpha,p}$ can be extended to a linear surjective isometry of $X_{\Sc_\alpha,p}$. Moreover, if $p\in \{1,2\}$, it is not necessary to assume that $(\Sc_\alpha)_{\alpha<\omega_1}$ is  good. 
	\end{theorem}
	
	In the sequel, we fix a  transfinite Schreier sequence $(\Sc_\alpha)_{\alpha<\omega_1}$, an ordinal $1\leqslant \alpha<\omega_1$ and $p\in [1,\infty)$.  Moreover, we assume that $(\Sc_\alpha)_{\alpha<\omega_1}$ is good when $p\in(1,\infty)\minus\{2\}$.
	
	\medskip
	We denote by $\Sp$ the unit sphere of $X_{\Sc_\alpha,p}$ and we fix a surjective isometry $T: \Sp\to\Sp$. Our aim is to show that $T$ can be extended to a linear surjective isometry $\Tilde{T}: X_{\Sc_\alpha,p}\to X_{\Sc_\alpha,p}$. 

     \medskip
    Fortunately, we know by~\cite{ABC},~\cite{BFT} (for $p=1$) and \cite{F} (for $p>1$) that any linear surjective isometry $J$ of $X_{\Sc_\alpha,p}$ is \emph{diagonal}, \textit{i.e.} there exists a sequence of signs $(\theta_i)_{i\in \N}$ (\textit{i.e.} $\theta_i=\pm 1$) such that  $Je_i=\theta_i e_i$ for all $i\in \N$. It follows that if $T$ can be extended to a linear surjective isometry of $X_{\Sc_\alpha,p}$, then there exists a sequence of signs $(\theta_i)$ such that $T(x)(i)=\theta_i x(i)$ for any $x\in \Sp$ and $i\in \N$. Conversely, this formula does define a linear surjective isometry of $X_{\Sc_\alpha,p}$.  Therefore, our aim  is to prove the following theorem.
	
	\begin{theorem}\label{main2}
		There exists a sequence of signs $(\theta_i)_{i\in \N}$ such that for any $x\in \Sp$ and $i\in \N$, we have $T(x)(i)=\theta_i x(i)$.    
	\end{theorem}
	If we are able to prove Theorem~\ref{main2}, then we  immediately obtain Theorem~\ref{main}. %, since $T$ is the restriction of the linear surjective isometry $\Tilde{T}:X_{\Sc_\alpha,p}\to X_{\Sc_\alpha,p}$ defined by $\Tilde{T}(e_i)=\theta_i e_i$ for any $i\in \N$. 
	The main step  to prove it is to show the following proposition. 
	
	\begin{proposition}\label{main3}
		There exists a sequence of signs $(\theta_i)_{i\in \N}$ such that $T(e_i)=\theta_i e_i$ for any $i\in \N$. 
	\end{proposition}
	Note that Proposition~\ref{main3} re-proves some theorems from~\cite{ABC},~\cite{BFT}  and ~\cite{F} concerning the description of linear surjective isometries of $X_{\Sc_\alpha,p}$, for $1\leqslant p<\infty$, and generalizes \cite[Proposition 3.2 (b)]{F} (in the real case) because we do not need to assume that $(\Sc_\alpha)_{\alpha<\omega_1}$ is good when $p=2$. 
	
	\smallskip\smallskip
	Before starting the proof of Theorem~\ref{main2}, we collect a few facts concerning the families $\Sc_\alpha$ for future use.
	
	\begin{fact}\label{max}\emph{(\cite[Lemma 3.1]{G})} 
		If  $G \in \Sc_\alpha \minus \Sc_\alpha^{MAX}$, then $G \cup \{l\}\in \Sc_\alpha$, for all $l > G$.
	\end{fact}
	\begin{fact}\label{maxsucc}\emph{(\cite[Lemma 3.9]{F})}
		Let $\alpha<\omega_1$ be a successor ordinal, $\alpha=\beta+1$. If $G \in \Sc_\alpha^{MAX}$, then $G= \bigcup_{i=1}^m G_i$ where $G_1<\dots<G_m \in \Sc_\beta^{MAX}$ and $m=\min G_1$.    
	\end{fact}
	\begin{fact}\label{Fact3}\emph{(\cite[Fact 3.4]{F})}
		Let $F\subset\N$, and let $m,n,p\in\N$. 
		If $m<n<p<F$ and $\{m\}\cup F \in \Sc_\alpha$, then $\{n,p\}\cup F \in \Sc_\alpha$.
	\end{fact} 
	In the remainder of this paper, we will prove Theorem~\ref{main2} for $p\in (1,\infty)$ in Section~\ref{sec1}, and for $p=1$ in Section~\ref{secp1}. %~\ref{sec1}, we collect some useful information on $T$ and the supports of $T(e_i)$, $i\in \N$, that are true for any $p\in (1,\infty)$. Then, we show Proposition~\ref{main3} in Section~\ref{sec2} for $p\in (1,\infty)\minus \{2\}$, and in Section~\ref{sec3} for $p=2$. Finally, we prove Theorem~\ref{main2} in Section~\ref{sec4}.

	\section{The case $p\in(1,\infty)$}\label{sec1}
	In this section, we will treat the case where $p\in (1,\infty)$. As mentioned in the introduction, we  fix a  good transfinite Schreier sequence $(\Sc_\alpha)_{\alpha<\omega_1}$, an ordinal $1\leqslant \alpha<\omega_1$ and $p\in (1,\infty)$. Moreover, we fix a surjective isometry $T:\Sp\to \Sp$ where $\Sp$ is the unit sphere of $X_{\Sc_\alpha,p}$. For simplicity, we will write $\Vert\,\cdot\,\Vert$ instead of $\Vert \,\cdot\,\Vert_{\Sc_\alpha,p}$. 
	
	\smallskip
	This section contains four subsections. In subsection~\ref{subzero}, we address facts that are true for any $p\in (1,\infty)$. Then, we will prove Proposition~\ref{main3} in Subsections~\ref{sec2} and~\ref{sec3} for $p\in (1,\infty)\minus \{2\}$ and $p=2$ respectively. Finally, we will prove Theorem~\ref{main2} in Subsection~\ref{sec4} for $p\in (1,\infty)$. 
	
	\subsection{General facts}\label{subzero}
	We begin this subsection with an important fact that will be used frequently in this section.
	\begin{fact}\label{normeatteinte}
		For any $x\in \Sp$,  there exists $F\in \Sc_\alpha$ such that 
		\[\sum_{i\in F}\vert x(i)\vert^p =1. \]
	\end{fact}
	\begin{proof}
		This is true by~\cite[Proposition 2.1]{ABC}. 
	\end{proof}
	
	Now, we address a collection of lemmas that work for $p\in (1,\infty)$. Obviously, the next lemma is not true for $p=1$, which is why we will treat the case $p=1$ separately in Section~\ref{secp1}.
	\begin{lemma}\label{l1}
		Let $n\in \N$ and $x\in \Sp$. Then, $\Vert x + e_n \Vert=2 $ if and only if $x(n)=1$. 
	\end{lemma}
	\begin{proof}
		We first prove the forward implication. By Fact~\ref{normeatteinte}, there exists $F \in \Sc_\alpha$ such that $\sum_{i\in F}\vert x(i)+e_n(i)\vert^p=2^p$. If $n\notin F$, then this sum is less than or equal to $1$, which is a contradiction. Hence $n \in F$. Moreover, if $\sum_{i\in F} \vert x(i)\vert ^p<1$, then by Minkowski inequality and since $n \in F$, we get 
		\[ \biggl(\sum_{i\in F}\vert x(i)+e_n(i)\vert^p\biggr)^{\frac1p}\leqslant \biggl(\sum_{i\in F}\vert x(i)\vert^p\biggr)^{\frac1p}+ \biggl(\sum_{i\in F}\vert e_n(i)\vert^p\biggr)^{\frac1p} <2, \]
		which is a contradiction. Hence, $\sum_{i\in F} \vert x(i)\vert ^p=1$. Therefore, 
		\[  \biggl(\sum_{i\in F}\vert x(i)+e_n(i)\vert^p\biggr)^{\frac1p}=2= \biggl(\sum_{i\in F}\vert x(i)\vert^p\biggr)^{\frac1p}+ \biggl(\sum_{i\in F}\vert e_n(i)\vert^p\biggr)^{\frac1p}.  \]
		So by Minkowski again (the equality case), there exists $\lambda \geqslant 0$ such that $x(i)= \lambda e_n(i)$ for every $i \in F$. This implies that $x(i)=0$ for every $i \in F\minus \{n\}$ and $x(n)=\lambda\geqslant 0$. Since  $\sum_{i\in F} \vert x(i)\vert ^p=1$, it follows that $x(n)=1$. The reverse implication is obvious. 
	\end{proof}
	
	The proof of the following lemma is similar to that of~\cite[Lemma 2.3]{T1}.
	
	\begin{lemma}\label{l2}
		We have $T(-e_1)=-T(e_1)$ and $T^{-1}(-e_1)=-T^{-1}(e_1)$.  
	\end{lemma}
	\begin{proof}
		Since $T$ is surjective, there exists $x_1\in \Sp$ such that $T(x_1)=-T(e_1)$. We need to show that $x_1=-e_1$. 
		
		Since $T$ is an isometry and $T(x_1)=-T(e_1)$, we have
		\[ \Vert x_1-e_1\Vert= \Vert T(x_1)-T(e_1)\Vert=2. \]
		So Lemma~\ref{l1} implies that  $x_1(1)=-1$. Now, let $j \in \N\minus \{1\}$ and let $x_j, y_j \in \Sp$ be such that $T(x_j)=-T(e_j)$ and $T(y_j)=-T(-e_j)$. Since $T$ is an isometry and by Lemma~\ref{l1} again, we get $x_j(j)=-1$ and $y_j(j)=1$. Since $x_j(j)=-1$, it follows that 
		\[\vert x_1(j)+1\vert \leqslant \Vert x_1-x_j\Vert = \Vert T(x_1)-T(x_j)\Vert=\Vert -T(e_1)+T(e_j)\Vert =\Vert e_j-e_1\Vert =1. \]
		In the same way (and since $y_j(j)=1$), we get $\vert x_1(j)-1\vert \leqslant 1$. Hence, $x_1(j)=0$. 
		So, we have shown that $x_1(1)=-1$ and $x_1(j)=0$ for every $j\in \N\minus \{1\}$, \textit{i.e.} $x_1=-e_1$.
		
		Similarly, we show that $T^{-1}(-e_1)=-T^{-1}(e_1)$.
	\end{proof}
	\begin{lemma}\label{l3}
		Let $u\in \Sp$. Then, $\min\{\Vert u+x\Vert, \Vert u-x\Vert\}\leqslant1$ for every $x \in \Sp$  if and only if $u\in \{\pm e_1\}$.
	\end{lemma}
	\begin{proof}
		The reverse implication is obvious because $\{1\}\in \Sc_\alpha^{MAX}$. For the forward implication, towards a contradiction, assume that there exists $i \geqslant 2$ such that $u(i)\neq 0$.  
		
		Assume first that there exists $j \geqslant 2$ such that $u(j)= 0$. Then $\{i,j\}\in \Sc_\alpha$ and $i\neq j$, so 
		\[\Vert u\pm e_j\Vert^p\geqslant \vert u(i)\vert^p+1 >1,\]
		which is a contradiction. 
		
		Assume now that $ u(j)\neq 0$, for every $j\geqslant 2$. Note that in this case, $\vert u(j)\vert \neq 1$  for every $j\geqslant 2$ since othewise we get $\Vert u \Vert >1$ because $\{m,n\}\in \Sc_\alpha$ for any $2\leqslant m<n \in \N$.  By Fact~\ref{normeatteinte}, there exists $F \in \Sc_\alpha$ such that $\sum_{k\in F}\vert u(k)\vert^p=1$. Let $b :=\max\{\vert u(k)\vert; \; k\in F\}$ and let $i_0 \in F $ be such that $\vert u(i_0)\vert =b$. Since $\Vert u-\sum_{k=1}^{n}u(k)e_k\Vert$ converges to zero as $n$ tends to infinity, one can find $j > F$ such that $\vert u(j)\vert < 1-b$. We have 
		\[ \Vert u+ \mathrm{sgn}(u(j))e_{j}\Vert\geqslant \vert u(j)+ \mathrm{sgn}(u(j))\vert=1+\vert u(j)\vert >1. \]
		Moreover, since $(F\minus \{i_0\})\cup \{j\}\in \Sc_\alpha$ (because $\Sc_\alpha $ is spreading), we have 
		\[\begin{aligned}
			\Vert u-\mathrm{sgn}(u(j))e_j\Vert^p&\geqslant \sum_{k\in F\minus \{i_0\}}\vert u(k)\vert^p+(1-\vert u(j)\vert)^p\\&>\sum_{k\in F\minus \{i_0\}}\vert u(k)\vert^p+b^p\\&=\sum_{k\in F}\vert u(k)\vert^p=1
		\end{aligned}\]
		which is a contradiction again. 
		
		Hence, $ u(i) =0$ for every $i \geqslant 2$, and so since $\Vert u \Vert =1$, we have  $u=\pm e_1$.
	\end{proof}
	\begin{lemma}\label{l4}
		There exists a sign $\theta_1$ such that $T(e_1)=\theta_1 e_1$ and $T^{-1}(e_1)=\theta_1 e_1$.   
	\end{lemma}
	\begin{proof}
		Since $\min\{\Vert e_1+x\Vert, \Vert e_1-x\Vert\}\leqslant1$ for every $x \in \Sp$ (by Lemma~\ref{l3}), $T$ is a surjective isometry and $T(-e_1)=-T(e_1)$ (by Lemma~\ref{l2}), it follows that $\min\{\Vert T(e_1)+y\Vert, \Vert T(e_1)-y\Vert\}\leqslant1$ for every $y \in \Sp$. Hence, there exists a sign $\theta_1$ such that $T(e_1)=\theta_1 e_1$ by Lemma~\ref{l3}, and this implies that $T^{-1}(e_1)=\theta_1 e_1$ by Lemma~\ref{l2}.
	\end{proof}
	
	For any $x\in \Sp$, we denote by $\mathrm{supp}(x)$ the support of $x$. Moreover, from now on, we denote by $A$ the set of all  $x\in \Sp$ such that $1\in \mathrm{supp}(x)$.
	
	\begin{lemma}\label{l5}
		We have $T(A)=A$. 
	\end{lemma}
	\begin{proof}
		Let $x\in A$ and $\varepsilon=-\mathrm{sgn}(x(1))$. Towards a contradiction, assume that $T(x)(1)=0$. We have 
		\[\Vert x-\varepsilon e_1\Vert \geqslant \vert x(1)-\varepsilon \vert= 1+\vert x(1)\vert >1.\]
		However, using Lemma~\ref{l2} and Lemma~\ref{l4}, we have
		\[\Vert T(x)-T(\varepsilon e_1)\Vert = \Vert T(x)-\varepsilon T( e_1)\Vert =\Vert T(x)-\varepsilon \theta_1 e_1\Vert=1,\]
		because $T(x)(1)=0$ and $\{1\}\in \Sc_\alpha^{MAX}$, which contradicts the fact that $T$ is an isometry. Hence $T(x)\in A$. We show in the same way that $T^{-1}(x)\in A$ for every $x\in A$.    
	\end{proof}
	\begin{lemma}\label{l6}
		For every $i \in \N\minus \{1\}$, we have $T(-e_i)=-T(e_i)$ and $T^{-1}(-e_i)=-T^{-1}(e_i)$.
	\end{lemma}
	\begin{proof}
		Let $i\in \N\minus\{1\}$. Since $T$ is surjective, there exists $x_i\in \Sp$ such that $T(x_i)=-T(e_i)$. We want to show that $x_i=-e_i$. Since $\Vert x_i-e_i\Vert=\Vert T(x_i)-T(e_i)\Vert=2$, it follows from Lemma~\ref{l1}  that $x_i(i)=-1$. This implies that for every $j\notin \{1,i\}$, we have $x_i(j)=0$ since otherwise we get $\Vert x_i\Vert>1$ (because $\{i,j\} \in \Sc_\alpha$). So it remains to show that $x_i(1)=0$, but this is clear by Lemma~\ref{l5} since $x_i=T^{-1}(-T(e_i))$ and $e_i\notin A$. So we have shown that $x_i(i)=-1$ and $x_i(j)=0$ for every $j\neq i$, \textit{i.e.} $x_i=-e_i$. In the same way, we show that $T^{-1}(-e_i)=-T^{-1}(e_i)$.%Towards a contradiction, assume that $x_i(1)\neq 0$ and let $\varepsilon=-\mathrm{sgn}(x_i(1))$. We have 
		%\[\Vert x_i-\varepsilon e_1\Vert = \vert x_i(1)-\varepsilon\vert=1+\vert x_i(1)\vert >1\].
		%However, by Lemma~\ref{l2} and Lemma~\ref{l4}, we get
		%\[\Vert T(x_i)-T(\varepsilon e_1)\Vert=\Vert-T(e_i)-T(\]
	\end{proof}
	For every $i\in\N$, we set 
	\[ T(e_i)=: f_i\qquad{\rm and}\qquad T^{-1}(e_i)=: d_i.\]
	\begin{lemma}\label{l7}
		For every $n\in \N\minus \{1\}$, $\mathrm{supp} (f_n)$ and $\mathrm{supp} (d_n)$ belong to  $\Sc_\alpha\minus\Sc_\alpha^{MAX}$. 
	\end{lemma}
	To prove Lemma~\ref{l7}, we need the following fact. 
	\begin{fact}\label{f1}
		Let $n \in\N\minus\{1\}$. For any $y \in \Sp\minus (A\cup \{f_n\})$, we have $\Vert f_n+y\Vert <2 $, and for any $z\in \Sp\minus (A\cup \{d_n\})$,   we have $\Vert d_n+z\Vert <2 $.  
	\end{fact}
	\begin{proof}[Proof of Fact~\ref{f1}]
		Let us show that $\Vert f_n+y\Vert <2 $ for any $y \in \Sp\minus (A\cup \{f_n\})$. Lemma~\ref{l1} implies that if $x\in \Sp\minus A$, then $\Vert x+e_n \Vert=2 $ if and only if $x=e_n$. It follows that if $y \in \Sp\minus A$, then $\Vert y+f_n\Vert=2$ if and only if $y=f_n$, since $T$ is a surjective isometry, $T(A)=A$ (by Lemma~\ref{l5}) and $T(-e_n)=-f_n$ (by Lemma~\ref{l6}). So we get the desired result since $\Vert f_n+y\Vert\leqslant\Vert f_n\Vert+ \Vert y\Vert=2$ for any $y\in \Sp$. In the same way we show that $\Vert d_n+z\Vert <2 $ for any $z\in \Sp\minus (A\cup \{d_n\})$.
	\end{proof}
	\begin{proof}[Proof of Lemma~\ref{l7}]
		Let $n \in \N\minus\{1\}$. We want to show that $\mathrm{supp} (f_n )\in \Sc_\alpha\minus\Sc_\alpha^{MAX}$. Since $\Vert f_n \Vert =1$, there exists $F \in \Sc_\alpha$ such that $\sum_{i \in F}\vert f_n(i)\vert^p=1$ and $f_n(i)\neq 0$ for all $i \in F$, by Fact~\ref{normeatteinte} and since $\Sc_\alpha$ is hereditary. We show first that $f_n(j)=0$ for all $j \notin F$. We know by Lemma~\ref{l5} that $f_n(1)=0$, so $1\notin F$. Towards a contradiction, assume that there exists $j \in \N \minus (F \cup \{1\})$ such that $f_n(j)\neq 0$ and let $y:=\sum_{i\in F}f_n(i)e_i$. It is clear that $y\in \Sp\minus A$. Moreover, 
		\[\Vert f_n+y\Vert \geqslant\biggl(\sum_{i\in F}\vert f_n(i)+ y(i)\vert^p\biggr)^{\frac1p}=2.\]
		Hence, $\Vert f_n+y\Vert =2$ which contradicts Fact~\ref{f1} since $y \neq f_n$. So we have shown that $f_n(j)=0$ for all $j \notin F$. Now, assume that $F\in \Sc_\alpha^{MAX}$. Let $j \in \N \minus (F \cup \{1\})$ (so $f_n(j)=0$) and let $z\in X_{\Sc_\alpha,p}$ be such that 
		\[
		z(i)= \begin{cases}
			f_n(i) &\text{if }i \in F, \\
			\min\{\vert f_n(i)\vert; \; i \in F\} &\text{if } i=j,\\
			0&\text{otherwise. }
		\end{cases}
		\]
		We want to show that $z\in \Sp$. Let $G\in \Sc_\alpha$. If $G=F$, then $\sum_{i\in G}\vert z(i)\vert^p=1$. If $j \notin G$, then the sum is less than or equal to 1. If $j \in G$, then there exists $i\in F$ such that $i\notin G$ because we are assuming that $F\in \Sc_\alpha^{MAX}$, hence the sum is also less than or equal to 1. Therefore, $\Vert z\Vert=1$. However, $\Vert z+f_n\Vert =2$ (because $\sum_{i\in F}\vert z(i)+f_n(i)\vert^p=2^p$ and $\Vert z+f_n\Vert\leqslant\Vert z\Vert + \Vert f_n\Vert=2$) which contradicts Fact~\ref{f1} since $z(j)\neq0$. In the same way, we show that $\mathrm{supp } (d_n) \in \Sc_\alpha\minus\Sc_\alpha^{MAX}$.
	\end{proof}
	\begin{lemma}\label{l8}
		Let $n,m\in \N\minus \{1\}$. Then, $f_n(m)\neq0 $ if and only if $d_m(n)\neq0$. 
	\end{lemma}
	\begin{proof}
		Let $m,n \in \N\minus \{1\}$ be such that $f_n(m)\neq0 $ and let $\varepsilon$ be the sign of $f_n(m)$. Since $m\in \mathrm{supp}(f_n)$, $\mathrm{supp}(f_n)\in \Sc_\alpha$ (by Lemma~\ref{l7}) and $p>1$,  it follows that  
		\[ \begin{aligned}
			\Vert f_n+ \varepsilon e_m\Vert^p&=\sum_{i\in \mathrm{supp }(f_n)}\vert f_n(i)+ \varepsilon e_m(i)\vert^p\\&=\sum_{i\in \mathrm{supp }(f_n)\minus \{m\}}\vert f_n(i)\vert^p+\vert f_n(m)+\varepsilon\vert^p\\&=\sum_{i\in \mathrm{supp }(f_n)\minus \{m\}}\vert f_n(i)\vert^p+(1+\vert f_n(m)\vert)^p\\&>\sum_{i\in \mathrm{supp }(f_n)\minus \{m\}}\vert f_n(i)\vert^p+\vert f_n(m)\vert^p+1=2.
		\end{aligned}\]
		However, using Lemma~\ref{l6}, we have 
		\[\begin{aligned}
			\Vert f_n+ \varepsilon e_m\Vert&=\Vert e_n-T^{-1}(-\varepsilon e_m)\Vert\\&=\Vert e_n+\varepsilon T^{-1}(e_m)\Vert\\&=\Vert e_n+\varepsilon d_m\Vert.
		\end{aligned}\]
		Hence, $\Vert e_n+\varepsilon d_m\Vert^p>2$. Now, towards a contradiction, assume that $n \notin \mathrm{supp}(d_m)$. Then, for any $G \in \Sc_\alpha$, we have
		\[ \sum_{k\in G}\vert e_n(k)+ \varepsilon d_m(k)\vert^p=1+ \sum_{k \in G\minus \{n\}}\vert d_m(k)\vert^p\leqslant  2\qquad\hbox{if $n\in G$}\] 
		and 
		\[ \sum_{k \in G}\vert e_n(k)+ \varepsilon d_m(k)\vert^p=\sum_{k\in G}\vert d_m(k)\vert ^p \leqslant  1\qquad\hbox{if $n\notin G$}.\]
		Therefore, $\Vert e_n + \varepsilon d_m \Vert^p\leqslant  2$, which is a contradiction. So, $d_m(n)\neq0$. The proof of the reverse implication is the same.
	\end{proof}
	\begin{lemma}\label{l9}
		For every $k \in \N\minus\{1\}$, there exists $n>k$ such that $k<\mathrm{supp}(f_n)$ and $k<\mathrm{supp}(d_n)$.   
	\end{lemma}
	\begin{proof}
		Let $k \in \N\minus\{1\} $. Towards a contradiction, assume that for all $n >k$, $\mathrm{supp}(f_n )\cap \{2, \dots ,k\}\neq\emptyset $ or  $\mathrm{supp}(d_n )\cap \{2, \dots ,k\}\neq\emptyset $. Assume, without loss of generality, that $\mathrm{supp}(f_n )\cap \{2, \dots ,k\}\neq\emptyset $ for infinitely many $n>k$. It follows that there exist $a \in \{2, \dots ,k\}$ and infinitely many $n>k$ such that $f_n(a)\neq 0$. Hence,  by Lemma~\ref{l8},  $d_a(n) \neq 0$ for infinitely many $n$, which is a contradiction since $d_a \in c_{00}$ by Lemma~\ref{l7}. %In the same way, \textcolor{red}{one} shows that there exists $m>k$ such that $\mathrm{supp}(d_m)>k$.
	\end{proof}
	
	\begin{lemma}\label{f2}
		Let $n \in \N\minus\{1\}$. If $k> \mathrm{supp}(f_n)$, then $d_k(n)=0$ and $\{n\}\cup\mathrm{supp}(d_k)\in \Sc_\alpha$. 
	\end{lemma}
	\begin{proof}
		Since $k>\mathrm{supp}(f_n)$, it follows that $f_n(k)=0$ and so $d_k(n)=0$ by Lemma~\ref{l8}. It remains to show that $\{n\}\cup\mathrm{supp}(d_k)\in \Sc_\alpha$. Since $\mathrm{supp}(f_n) \in \Sc_\alpha\minus  \Sc_\alpha^{MAX}$ (by Lemma~\ref{l7}) and $k>\mathrm{supp}(f_n)$, it follows by Fact~\ref{max} that $\mathrm{supp}(f_n) \cup \{k\}\in \Sc_\alpha$. This implies that $\Vert f_n - e_k\Vert ^p=2$, and so $\Vert e_n-d_k\Vert^p = 2$. Hence, $\{n\}\cup\mathrm{supp}(d_k)\in \Sc_\alpha$ since otherwise we get $\Vert e_n - d_k\Vert^p<2$ which is a contradiction. 
	\end{proof}
	\begin{lemma}\label{imp}
		Let $\{i_1, \dots,i_n\}\in \Sc_\alpha$ be such that $f_{i_1}=\varepsilon_1 e_{j_1},\dots , f_{i_n}=\varepsilon_n e_{j_n}$ where $\varepsilon_1,\dots, \varepsilon_n$ are signs and $\{j_1,\dots,j_n\}\in \Sc_\alpha$. Then, for any $a_1,\dots,a_n\in \R\minus\{0\}$ such that $\sum_{k=1}^n\vert a_k\vert^p=1$, and for any $x,y\in \Sp$ such that $x(i_k)=a_k=y(j_k)$ for all $k \in \{1,\dots, n\}$, we have 
		\[  \begin{aligned}\vert T(x)(j_l)\vert=\vert a_l\vert \qquad\text{and}\qquad \vert T^{-1}(y)(i_l)\vert=\vert a_l\vert \end{aligned}\]   for all $l\in \{1,\dots, n\}$.
	\end{lemma}
	\begin{proof}
		Let $x\in \Sp$ be such that $x(i_k)=a_k$ for all $k\in \{1,\dots, n\}$, and let us fix $l\in \{1,\dots,n\}$. We denote by $\varepsilon$ the sign of $a_l$. Let us first compute $\Vert x+\varepsilon e_{i_l}\Vert ^p$. Let $G\in\Sc_\alpha$. If $i_l\notin G$, then $\sum_{k\in G}\vert x(k)+ \varepsilon e_{i_l}(k)\vert^p\leqslant1$. If $i_l\in G$, then \[\begin{aligned}
			\sum_{k\in G}\vert x(k)+\varepsilon e_{i_l}(k)\vert^p&=\sum_{k\in G\minus \{i_l\}}\vert x(k)\vert^p+\vert \varepsilon + x(i_l)\vert^p\\&=\sum_{k\in G\minus \{i_l\}}\vert x(k)\vert^p+ (1+ \vert a_l\vert)^p\\&=\sum_{k\in G}\vert x(k)\vert^p-\vert a_l\vert^p+ (1+ \vert a_l\vert)^p\\&\leqslant 1-\vert a_l\vert^p+ (1+ \vert a_l\vert)^p.
		\end{aligned}\]
		If we take $G=\{i_1,\dots, i_n\}$, then $\sum_{k\in G}\vert x(k)\vert^p=\sum_{k=1}^n\vert a_k\vert^p=1$, and hence  \[\sum_{k\in G}\vert x(k)+\varepsilon e_{i_l}(k)\vert^p= 1-\vert a_l\vert^p+ (1+ \vert a_l\vert)^p.\] So, \[\Vert x+\varepsilon e_{i_l}\Vert ^p= 1-\vert a_l\vert^p+ (1+ \vert a_l\vert)^p.\]
		Since $T$ is an isometry, $T(-e_{i_l})=-f_{i_l}$ by Lemma~\ref{l6} and $f_{i_l}=\varepsilon_l e_{j_l}$, it follows that  
		\[\Vert T(x)+\varepsilon  \varepsilon_l e_{j_l}\Vert^p= 1-\vert a_l\vert^p+ (1+ \vert a_l\vert)^p \geqslant 2.\]
		Let $F _0\in \Sc_\alpha$ be such that  \[\Vert T(x)+\varepsilon  \varepsilon_l e_{j_l}\Vert^p=\sum_{k\in F_0}\vert T(x)(k)+\varepsilon  \varepsilon_l e_{j_l}(k)\vert^p.\]
		If $j_l\notin F_0$, then $\Vert T(x)+\varepsilon  \varepsilon_l e_{j_l}\Vert^p\leqslant1 $, which is a contradiction. Hence $j_l\in F_0$ and we have 
		\[\begin{aligned}
			\sum_{k\in F_0}\vert T(x)(k)+\varepsilon  \varepsilon_l e_{j_l}(k)\vert^p&=\sum_{k\in F_0\minus \{j_l\}}\vert T(x)(k)\vert^p+ \vert T(x)(j_l)+\varepsilon \varepsilon_l\vert^p\\&= \sum_{k\in F_0}\vert T(x)(k)\vert^p-\vert T(x)(j_l)\vert ^p+ \vert T(x)(j_l)+\varepsilon \varepsilon_l\vert^p\\&\leqslant 1-\vert T(x)(j_l)\vert ^p+ \vert T(x)(j_l)+\varepsilon \varepsilon_l\vert^p\\&\leqslant  1-\vert T(x)(j_l)\vert ^p+ (1+\vert T(x)(j_l)\vert)^p.   
		\end{aligned}\]
		Hence, \[1-\vert a_l\vert^p+ (1+ \vert a_l\vert)^p\leqslant  1-\vert T(x)(j_l)\vert ^p+ (1+\vert T(x)(j_l)\vert)^p, \]  
		so
		\[(1+ \vert a_l\vert)^p -\vert a_l\vert^p\leqslant (1+\vert T(x)(j_l)\vert)^p -\vert T(x)(j_l)\vert ^p.\]
		Since the function $t \mapsto (1+t)^p-t^p $ is strictly increasing on $[0,1]$, it follows that 
		\[\vert T(x)(j_l)\vert\geqslant\vert a_l\vert. \]
		This holds for any $l\in \{1,\dots,n\}$. So, since $\{j_1,\dots, j_n\}\in \Sc_\alpha$, we have  
		\[\Vert T(x)\Vert ^p\geqslant \sum_{l=1}^n \vert T(x)(j_l)\vert^p\geqslant \sum_{l=1}^n\vert a_l\vert^p=1.\] 
		If there exists $l\in \{1, \dots, n\}$ such that $\vert T(x)(j_l)\vert>\vert a_l\vert$, then $\Vert T(x)\Vert ^p> \sum_{l=1}^n\vert a_l\vert^p=1$ which is a contradiction since $T(x)\in \Sp$. Therefore, $\vert T(x)(j_l)\vert=\vert a_l\vert$ for every $l\in \{1,\dots, n\}$. %It remains to show that $\mathrm{sgn}(T(x)(j_l))=\varepsilon\varepsilon_l$ where $\varepsilon $ is the sign of $a_l$. Towards a contradiction, assume that $\mathrm{sgn}(T(x)(j_l))=-\varepsilon\varepsilon_l$. Then 
		%\[\begin{aligned}
			%   \Vert T(x)+\varepsilon  \varepsilon_l e_{j_l}\Vert^p&=\sum_{k\in F_0}\vert T(x)(k)+\varepsilon  \varepsilon_l e_{j_l}(k)\vert^p\\&=\sum_{k\in F_0\minus \{j_l\}}\vert T(x)(k)\vert^p+ \vert T(x)(j_l)+\varepsilon \varepsilon_l\vert^p\\&\leqslant 1+1=2,
			%\end{aligned}\]
			% which contradicts (\ref{truc}). Therefore, $\mathrm{sgn}(T(x)(j_l))=\varepsilon\varepsilon_l$ and $T(x)(j_l)=\varepsilon_l a_l$.
			
			In the same way, we show that $\vert T^{-1}(y)(i_l)\vert=\vert a_l\vert$ for any  $y\in \Sp$ such that $y(j_k)=a_k$ for all $k \in \{1,\dots, n\}$ because $f_{i_1}=\varepsilon_1 e_{j_1}, \dots , f_{i_n}=\varepsilon_n e_{j_n}$ imply that $d_{j_1}=\varepsilon_1e_{i_1}, \dots, d_{j_n}=\varepsilon_n e_{i_n}$ by Lemma~\ref{l6}.
		\end{proof}

		\subsection{The case $p\neq2$}\label{sec2}
		In this subsection, we assume that $p\in (1,\infty)\minus \{2\}$. Recall that $(\Sc_\alpha)_{\alpha<\omega_1}$ is good. The aim of this subsection is to prove Proposition~\ref{main3} in this case.
		
		\medskip The proof of the following lemma requires $p\neq 2$.
		
		\begin{lemma}\label{l10}
			For any $k \in \N\minus\{1\}$, there exist $i,j>k$ such that $f_i=\pm e_j$ and  $d_j=\pm e_i$.
		\end{lemma}
		\begin{proof}
			By Lemma~\ref{l9}, there exists $i>\{k\}\cup\mathrm{supp}(f_2)\cup \mathrm{supp}(f_3)$, such that $\mathrm{supp }(f_i)> \{k\}\cup\mathrm{supp}(f_2)\cup \mathrm{supp}(f_3)$. We need to show that there exists $j $ such that $f_i=\pm e_j$. Towards a contradiction, assume that $\mathrm{supp}(f_i)$ contains at least two distinct integers $r$ and $s$. Since $f_i(r)\neq0$ and $f_i(s)\neq0$, it follows from Lemma~\ref{l8} that $d_r(i)\neq 0$ and $d_s(i)\neq 0$ (\textit{i.e.} $i\in \mathrm{supp}(d_r)\cap \mathrm{supp}(d_s)$). Since $r,s > \mathrm{supp}(f_2)\cup \mathrm{supp}(f_3) $ (because $r,s\in \mathrm{supp}(f_i)$), it follows from Lemma~\ref{f2} that $d_r(2)=d_s(2)=d_r(3)=d_s(3)=0$ (and we know by Lemma~\ref{l5} that $d_r(1)=d_s(1)=0$), and that $\{2\}\cup\mathrm{supp}(d_r) \in \Sc_\alpha$ and $\{2\}\cup\mathrm{supp}(d_s) \in \Sc_\alpha$. Now, we need the following claim:
			\begin{claim}\label{cl1}
				We have $\mathrm{supp}(d_r)\cup \mathrm{supp}(d_s)\in \Sc_\alpha $.
			\end{claim}
			\begin{proof}[Proof of Claim~\ref{cl1}]
				
				Assume that $\alpha$ is a successor ordinal ($\alpha=\beta+1$ where $\beta$ is an ordinal). Since $\{2\}\cup\mathrm{supp}(d_r) \in \Sc_\alpha$ and $\{2\}\cup\mathrm{supp}(d_s) \in \Sc_\alpha$, it follows that $\mathrm{supp}(d_r)=F_1\cup F_2$ and $\mathrm{supp}(d_s)=F_3\cup F_4$ where $F_1, F_2, F_3, F_4 \in \Sc_\beta$. Since $\min(\mathrm{supp}(d_r)\cup\mathrm{supp}(d_s))\geqslant 4$, it follows that $F_1\cup F_2 \cup F_3\cup F_4\in \Sc_\alpha$. 
				
				Now, assume that $\alpha$ is a limit ordinal (the approximating sequence $(\alpha_n)$ is such that $\alpha_n=\beta_n+1$ for all $n\in \N$ and $(S_{\beta_n})$ is increasing because  $(\Sc_\alpha)_{\alpha<\omega_1}$ is good).  Since $\{2\}\cup\mathrm{supp}(d_r) \in \Sc_\alpha$ and $\{2\}\cup\mathrm{supp}(d_s) \in \Sc_\alpha$, there exist $k_1,k_2\leqslant2$ such that $\{2\}\cup\mathrm{supp}(d_r) \in\Sc_{\alpha_{k_1}}$ and $\{2\}\cup\mathrm{supp}(d_s) \in \Sc_{\alpha_{k_2}}$. Hence, $\{2\}\cup\mathrm{supp}(d_r)\in  \Sc_{\alpha_{k'}}$ and $\{2\}\cup\mathrm{supp}(d_s)\in  \Sc_{\alpha_{k'}}$ where $k'=\max\{k_1,k_2\}$ since $(\Sc_{\alpha_n})$ is increasing (because $(\Sc_{\beta_n})$ is increasing). So, $\mathrm{supp}(d_r)=F'_1\cup F'_2$ and $\mathrm{supp}(d_s)=F'_3\cup F'_4$ where $F'_1, F'_2, F'_3, F'_4\in \Sc_{\beta_{k'}}$.  Since $\min(\mathrm{supp}(d_r)\cup\mathrm{supp}(d_s))\geqslant 4$, it follows that $F'_1\cup F'_2 \cup F'_3\cup F'_4\in \Sc_{\alpha_{k'}}$ and so $\mathrm{supp}(d_r)\cup\mathrm{supp}(d_s)\in \Sc_{\alpha}$ (because $k'<\mathrm{supp}(d_r)\cup\mathrm{supp}(d_s)$).
			\end{proof}
			Since $\mathrm{supp}(d_r)\cup\mathrm{supp}(d_s)\in \Sc_{\alpha}$, it follows by Lemma~\ref{l6} that 
			\[\begin{aligned}\Vert d_r+d_s\Vert^p_{\ell_p}+\Vert d_r-d_s\Vert^p_{\ell_p}&=\Vert d_r+d_s\Vert^p+\Vert d_r-d_s\Vert^p\\&=\Vert e_r+e_s\Vert^p+\Vert e_r-e_s\Vert^p\\&=2+2\\&=2(\Vert d_r\Vert^p_{\ell_p}+\Vert d_s\Vert^p_{\ell_p}),\end{aligned}\]
			which is a contradiction, since $p\neq 2$ and $\mathrm{supp}(d_r)\cap \mathrm{supp}(d_s)\neq\emptyset$ (see e.g.~\cite[Theorem~8.3]{Car}). So, $\mathrm{supp}(f_i)=\{j\}$ where $j>k$, \textit{i.e.} $f_i=\pm e_j$. This means that $d_j=\pm e_i$ (by Lemma~\ref{l6}).    \end{proof}
		%\begin{lemma}\label{l11}
		% Let $\alpha=1$. For any $k \geqslant 2$, there exist $i,j>k$ such that $f_i=\pm e_j$. (Equivalently, for any $k \geqslant 2$, there exist $i,j>k$ such that $d_j=\pm e_i$.)
		% \end{lemma}
	% \begin{proof}
		%By Lemma~\ref{l9}, there exists $j>\{k\}\cup\mathrm{supp}f_2$ such that $\mathrm{supp }d_j > \{k\}\cup\mathrm{supp}f_2$.  Lemma~\ref{f2} implies that $d_j(2)=0$ and $\{2\}\cup\mathrm{supp}d_j \in \Sc_1$. Hence, $\mathrm{supp}d_j=\{i\}$ for some $i\geqslant 3$ (and $i>k$  because $\mathrm{supp}d_j>k$). Therefore, $d_j=\pm e_i$ which means that $f_i=\pm e_j$ by Lemma ~\ref{l6}.  
		% \end{proof}
	
	% \textcolor{red}{In what follows, we consider $p\neq 2$ and $1\leqslant\alpha<\omega_1$ or $p=2$ and $\alpha=1$.}
	
	\begin{lemma}\label{equality}
		Let $k,k'\in \N\minus \{1\}$ be such that $f_k=\pm e_{k'}$.  Then $k=k'$. 
	\end{lemma}
	\begin{proof}
		Towards a contradiction, assume that $k'>k$. Let $k_1:=k'+1$.  By Lemma~\ref{l10}, we can fix  $n_1,n'_1>k_1$ such that $f_{n_1}=\pm e_{n'_1}$. Next, let $k_2:=\max\{n_1,n'_1\}$ and fix $n_2,n'_2>k_2$ such that $f_{n_2}=\pm e_{n'_2}$. Continuing in this way, we obtain three strictly increasing sequences $(k_j)_{j=1}^\infty$, $(n_j)_{j=1}^\infty$ and $(n'_j)_{j=1}^\infty$.  Now, we will distinguish two cases. 
		
		\smallskip  \textbf{Case 1:}  $\alpha$ is a successor ordinal ($\alpha=\beta
		+1)$. Since $k_1>2 $, it follows that $\{k_1\}\in \Sc_\beta^{MAX}$ if $\beta=0$ and that $\{k_1\}\in \Sc_\beta\minus\Sc_\beta^{MAX}$ if $\beta\geqslant 1$. Hence, Fact~\ref{max}, together with the fact that $\Sc_\beta $ is a compact family of finite subsets of $\N$, implies that there exists $r_1\geqslant 1$ such that $E_1:=\{k_j;\; j=1,\dots , r_1\}\in \Sc_\beta^{MAX}$. Let $F_1:=\{n_j;\; j=1,\dots , r_1\}$ and $F'_1:=\{n'_j;\; j=1,\dots , r_1\}$. Note that $F_1,F'_1\in \Sc_\beta$ since $E_1\in \Sc_\beta $ and $\Sc_\beta $ is spreading. In the same way, for every $i\in \{2,\dots, k+1\}$, there exists $r_i\geqslant r_{i-1}+1$  such that  $E_i:=\{k_j;\; j=r_{i-1}+1,\dots , r_i\}\in \Sc_\beta^{MAX}$, and then $F_i:=\{n_j;\; j=r_{i-1}+1,\dots , r_i\}$ and $F'_i:=\{n'_j;\; j=r_{i-1}+1,\dots , r_i\}$ belong to $\Sc_\beta$.\\
		So, we have $E_1<\dots<E_k<E_{k+1}$, $F_1<\dots<F_k<F_{k+1}$,  $F'_1<\dots<F'_k<F'_{k+1}$ and $\vert E_i\vert =\vert F_i\vert= \vert F'_i\vert$ for every $i \in \{1,\dots, k+1\}$.\\
		Now, let $a_1,\dots, a_{r_k}\in \R\minus \{0\}$ be such that $\sum_{i=1}^{r_k}\vert a_i\vert^p=1$. We denote $\sum_{i=1}^{r_k}a_ie_{n_i}$ by $x$. Then $x\in \Sp$ since $\sum_{i=1}^{r_k}\vert a_i\vert^p=1$ and $ \bigcup_{i=1}^k F_i\in \Sc_\alpha$ because $F_1<\dots<F_k\in \Sc_\beta$ and $\min \bigl( \bigcup_{i=1}^k F_i\bigr)=n_1>k $. Since $ \bigcup_{i=1}^k F_i\in \Sc_\alpha$, $\bigcup_{i=1}^k F'_i\in \Sc_\alpha$, and $f_{n_i}=\pm e_{n'_i}$ for all $i\in \{1,\dots, r_k\}$, it follows by Lemma~\ref{imp} that $\vert T(x)(n'_i)\vert=\vert a_i\vert $ for all $i\in \{1,\dots, r_k\}$. \\
		We claim that $T(x)(k')=0$. Indeed, if we assume that  $T(x)(k')\neq0$ and since $\{k'\}\cup \bigcup_{i=1}^k F'_i \in \Sc_\alpha$ (because $F'_1<\dots<F'_k \in \Sc_\beta$ and $k+1\leqslant k'$), then 
		\[\begin{aligned}
			\Vert T(x)\Vert^p &\geqslant \vert T(x)(k')\vert^p+ \sum_{i=1}^{r_k}\vert T(x)(n'_i)\vert^p\\&= \vert T(x)(k')\vert^p+ \sum_{i=1}^{r_k}\vert a_i\vert^p>1,  \end{aligned}\]
		which is a contradiction.\\ Since $f_k=\pm e_{k'}$, $\{k'\}\cup \bigcup_{i=1}^k F'_i \in \Sc_\alpha$ and $T(x)(k')= 0$, it follows that 
		\begin{align}\label{cont1}
			\Vert f_k-T(x) \Vert^p\geqslant 1+\sum_{i=1}^{r_k}\vert T(x)(n'_i)\vert^p= 1+\sum_{i=1}^{r_k}\vert a_i\vert^p=2.
		\end{align}
		Now, let us show that  $\Vert e_k-x\Vert^p<2 $. Fix $F \in \Sc_\alpha$.
		Note that if $k\notin F$ or if there exists $i\in\{1,\dots, r_k\}$ such that $n_i\notin F$, then 
		\begin{align}\label{con2} \sum_{i\in F}\vert e_k(i)-x(i)\vert ^p<2.\end{align}
		Now, we need the following claim. 
		
		\begin{claim}\label{cl2}
			We have $\{k,n_1,\dots,n_{r_k}\}\notin \Sc_\alpha$.
		\end{claim}
		\begin{proof}[Proof of Claim~\ref{cl2}]
			Towards a contradiction, assume that $\{k,n_1,\dots,n_{r_k}\}\in \Sc_\alpha$, and let $G \in \Sc_\alpha^{MAX}$ be such that   $k\in G$ and $n_i\in G$ for every $i \in\{1,\dots, r_k\}$. By Fact~\ref{maxsucc}, we may write $G = \bigcup_{i=1}^m G_i$ where $G_1< \dots< G_m \in \Sc_\beta^{MAX}$ and $\min G_1 =m$.  Since $k\in G$, it follows that $\min G \leqslant k$ and so $m=\min G_1\leqslant k$. If $m=1$ then $G=\{1\}$ because $\{1\}\in \Sc_\alpha^{MAX}$ which is a contradiction. Hence $m\geqslant 2$.  We claim that there exists $j_1\in \{1,\dots , r_1\}$ such that $n_{j_1}\notin G_1$. Indeed, assume towards a contradiction that $n_i\in G_1$ for all $i\in \{1,\dots , r_1\}$. Hence, $\{\min G_1,n_1, \dots ,n_{r_1}\}\subset G_1\in \Sc_\beta$, which implies that $\{\min G_1,n_1, \dots, n_{r_1}\}\in \Sc_\beta$ since $\Sc_\beta $ is hereditary. However, since $\{k_1,\dots, k_{r_1}, k_{r_1+1}\}\notin \Sc_\beta $ (because $\{k_1,\dots, k_{r_1}\}\in \Sc_\beta^{MAX}$) and since $\Sc_\beta$ is spreading, it follows that $\{\min G_1,n_1, \dots, n_{r_1}\}\notin \Sc_\beta$ (because $\min G_1\leqslant k<k_1$ and $n_i\leqslant k_{i+1}$ for all $i\in \N$ by construction) which is the desired contradiction. %This shows that if $m=1$, then there is no $G \in \Sc_\alpha^{MAX}$ such that $k\in G$ and $n_i\in G$ for every $i \in\{1,\dots, r_k\}$. \textcolor{red}{Now,} if $m\geqslant2$ and  
			Since we are assuming that $n_i\in G$ for every $i \in\{1,\dots, r_k\}$, and since we have shown that there exists $j_1\in \{1,\dots , r_1\}$ such that $n_{j_1}\notin G_1$, we must have $\min G_2\leqslant n_{r_1}$ and so $\min G_2\leqslant k_{r_1+1}$. One shows in the same way that for every $i\in \{1, \dots, m\}$, there exists $j_i\in \{r_{i-1}+1,\dots , r_i\}$ such that $n_{j_i}\notin G_i$ and  $\min G_i\leqslant n_{r_{i-1}}$. In particular, $\min G_m\leqslant n_{r_{m-1}}$ and  there exists $j_m\in \{r_{m-1}+1,\dots , r_m\}$ such that $n_{j_m}\notin G_m$. Hence, $n_{j_m}\notin G_i $ for all  $i\in \{1,\dots, m-1\}$ since $G_1<\dots<G_m$, $\min G_m \leqslant n_{r_{m-1}}$ and $n_{r_{m-1}}< n_{j_m}$. This implies that $n_{j_m}\notin G$, which is again a contradiction. 
		\end{proof}
		
		Claim~\ref{cl2} and (\ref{con2}) imply that  $\Vert e_k-x\Vert ^p<2$  which contradicts (\ref{cont1}) since $T$ is an isometry. So, $k'\leqslant k$. 
		
		\smallskip  \textbf{Case 2:} $\alpha$ is a limit ordinal (the approximating sequence $(\alpha_n)$ is such that $\alpha_n=\beta_n+1$ for all $n\in \N$ and $(S_{\beta_n})$ is increasing because  $(\Sc_\alpha)_{\alpha<\omega_1}$ is good). We will proceed as in the successor case but replacing $\beta$ with $\beta_k$. In other words, there exists $r_1> 1$ such that $E_1:= \{k_j;\; j=1,\dots, r_1\}\in \Sc_{\beta_k}^{MAX}$ and hence $F_1:= \{n_j;\; j=1,\dots, r_1\}$ and $F'_1:= \{n'_j;\; j=1,\dots, r_1\}$ belong to $\Sc_{\beta_k}$ since $\Sc_{\beta_k}$ is spreading.   In the same way, for every $i\in \{2,\dots, k+1\}$, there exists $r_i> r_{i-1}+1$  such that  $E_i:=\{k_j;\; j=r_{i-1}+1,\dots , r_i\}\in \Sc_{\beta_k}^{MAX}$, and then $F_i:=\{n_j;\; j=r_{i-1}+1,\dots , r_i\}$ and $F'_i:=\{n'_j;\; j=r_{i-1}+1,\dots , r_i\}$ belong to $\Sc_{\beta_k}$.  Now, let $a_1,\dots, a_{r_k}\in \R\minus \{0\}$ be such that $\sum_{i=1}^{r_k}\vert a_i\vert^p=1$ and  $y:= \sum_{i=1}^{r_k}a_ie_{n_i}$. Since $\bigcup_{i=1}^k F_i\in \Sc_{\alpha_k}$ and $k<\min \bigl( \bigcup_{i=1}^k F_i\bigr)=n_1$, it follows that $\bigcup_{i=1}^k F_i\in \Sc_{\alpha}$ and so $y \in \Sp$. In the same way, we have that $\bigcup_{i=1}^k F'_i\in \Sc_\alpha$.  Since $ \bigcup_{i=1}^k F_i\in \Sc_\alpha$, $\bigcup_{i=1}^k F'_i\in \Sc_\alpha$, and $f_{n_i}=\pm e_{n'_i}$ for all $i\in \{1,\dots, r_k\}$, it follows by Lemma~\ref{imp} that $\vert T(y)(n'_i)\vert=\vert a_i\vert $ for all $i\in \{1,\dots, r_k\}$. Now, it is easy to prove that  $T(y)(k')=0$, as we did for $x$ in the successor case, since $\{k'\}\cup \bigcup_{i=1}^k F'_i \in \Sc_{\alpha}$ (because $\{k'\}\cup \bigcup_{i=1}^k F'_i \in \Sc_{\alpha_k}$ and $k<\min \bigl( \{k'\}\cup \bigcup_{i=1}^k F'_i\bigr)=k'$). %and $k<\min \bigl( \{j\}\cup \bigcup_{i=1}^k F'_i\bigr)=j$, it follows that $\{j\}\cup \bigcup_{i=1}^k F'_i \in \Sc_\alpha$. Hence, if $T(y)(j)=0$, then 
		% \[\begin{aligned}
			% \Vert T(y)\Vert^p &\geqslant \vert T(y)(j)\vert^p+ \sum_{i=1}^{r_k}\vert T(y)(n'_i)\vert^p\\&= \vert T(y)(j)\vert^p+ \sum_{i=1}^{r_k}\vert a_i\vert^p>1,  \end{aligned}\]
		% which is a contradiction. So, $T(y)(j)=0$.\\ 
		Since $f_k=\pm e_{k'}$, $\{k'\}\cup \bigcup_{i=1}^k F'_i \in \Sc_\alpha$ and $T(y)(k')= 0$, it follows that 
		\begin{align}\label{contrad1}
			\Vert f_k-T(y) \Vert^p\geqslant 1+\sum_{i=1}^{r_k}\vert T(y)(n'_i)\vert^p= 1+\sum_{i=1}^{r_k}\vert a_i\vert^p=2.
		\end{align}
		Now, let us show that $\Vert e_k-y\Vert^p<2 $ which will contradict (\ref{contrad1}) since $T$ is an isometry. Fix $F \in \Sc_\alpha$.
		Note that if $k\notin F$ or if there exists $i\in\{1,\dots, r_k\}$ such that $n_i\notin F$, then 
		\begin{align}\label{contrad2} \sum_{i\in F}\vert e_k(i)-x(i)\vert ^p<2.\end{align}
		Now, let us show that $\{k,n_1,\dots,n_{r_k}\}\notin \Sc_\alpha$. Towards a contradiction, assume that there exists $G\in \Sc_\alpha^{MAX}$ such that $k\in G$ and $n_i\in G$ for every $i \in\{1,\dots, r_k\}$. Since $k\in G$, there exists $n \leqslant  \min G \leqslant  k$ such that $G \in \Sc_{\alpha_n}$.  
		Note that $G \in \Sc_{\alpha_n}^{MAX}$. Indeed, if $G \in \Sc_{\alpha_n}\minus \Sc_{\alpha_n}^{MAX}$, then Lemma~\ref{max} implies that  $G \cup \{l\} \in \Sc_{\alpha_n}$, for each $l>G$. Since $ n \leqslant  \min G= \min(G\cup\{l\})$, it follows that $G\cup\{l\} \in \Sc_\alpha$, which contradicts the fact that $G \in \Sc_\alpha^{MAX}$. So, by Fact~\ref{maxsucc}, we may write $G = \bigcup_{i=1}^m G_i$ where $G_1< \dots< G_m \in \Sc_{\beta_n}^{MAX}$ and $\min G_1 =m$. Since $\Sc_{\beta_n}\subset \Sc_{\beta_k}$ (because $n \leqslant  k$ and $(\Sc_{\beta_n})_{n\in \N}$ is increasing), it follows that $G_1,\dots, G_m\in \Sc_{\beta_k}$. So, we continue 
		exactly as the proof of Claim~\ref{cl2} in the successor case but replacing $\beta$ with $\beta_k$. 
		
		Hence,  $\Vert e_k-x\Vert^p<2$,  which is the desired contradiction. So, $k'\leqslant k$.

		\medskip \smallskip   Finally, Lemma~\ref{l6} implies that $d_{k'}=\pm e_k$, so we get in the same way a contradiction if we suppose that $k>k'$. Therefore, $k=k'$. 
	\end{proof}  
	
	\begin{lemma}\label{l12}
		Let $k\in \N\minus \{1\}$. We have $k\leqslant \mathrm{supp}(f_k)$ and $k\leqslant \mathrm{supp}(d_k)$. 
	\end{lemma}
	\begin{proof}
		Let us show that $k\leqslant \mathrm{supp}(f_k)$. Towards a contradiction, assume that $k_0:=\min \mathrm{supp}(f_k)<k$. If $f_k=\pm e_{k_0}$, then $k=k_0$, by Lemma~\ref{equality}, which is a contradiction. Hence there exists $m_0>k_0$ such that $m_0 \in \mathrm{supp}(f_k)$. By Lemmas~\ref{l10} and $\ref{equality}$, one can fix $k_1>\max(\mathrm{supp}(f_k)\cup\{k\})$ such that $f_{k_1}=\pm e_{k_1}$. In the same way, for every $n\in \N\minus \{1\}$, one can fix $k_n>k_{n-1}$ such that $f_{k_n}=\pm e_{k_n}$. Since $k_0>1$ (because $f_k(1)=0$ by Lemma~\ref{l5}), it follows that $\{k_0\}\in \Sc_\alpha\minus \Sc_\alpha^{MAX}$, and so there exists $r\geqslant1$ such that $\{k_0, k_1, \dots, k_r\}\in \Sc_\alpha^{MAX}$ by Fact~\ref{max}. Now let $a_1,\dots,a_r\in \R\minus \{0\}$ be such that $\sum_{i=1}^r\vert a_i \vert ^p=1$ and let $x:= \sum_{i=1}^r a_i e_{k_i}$. Note that $x \in \Sp$ since $\{k_1,\dots, k_r\}\in \Sc_\alpha$. We  know by Lemma~\ref{imp} that $\vert T(x)(k_i)\vert =\vert a_i \vert $, for every $i\in \{1,\dots,r\}$. Note that for any $j\geqslant k_0$ such that $j\notin \{k_1, \dots, k_r\}$, we have $\{j,k_1,\dots,k_r\}\in \Sc_\alpha$ since $\{k_0,k_1,\dots,k_r\}\in \Sc_\alpha$ and $\Sc_\alpha $ is spreading. This implies that for any such $j$ we have $T(x)(j)=0$ since otherwise we get 
		\[\begin{aligned}
			\Vert T(x)\Vert^p&\geqslant\vert T(x)(j)\vert^p+ \sum_{i=1}^r\vert T(x)(k_i) \vert ^p\\&=\vert T(x)(j)\vert^p+\sum_{i=1}^r\vert a_i \vert ^p>1,
		\end{aligned}\]
		which is a contradiction. \\
		Since $\{k, k_1,\dots, k_r\}\in \Sc_\alpha$ (because $\Sc_\alpha$ is spreading and we are assuming that $k>k_0$) and since  $k<k_1<\dots<k_r$ by construction, it follows that $\Vert e_k-x\Vert^p=2$. Hence $\Vert f_k-T(x)\Vert^p=2$. So, let $G \in \Sc_\alpha$ be such that $\sum_{i\in G}\vert f_k(i)-T(x)(i)\vert ^p=2$. \\
		Assume that there exists $i_0\in \mathrm{supp}(f_k)$ such that $i_0\notin G$. Then, since $T(x)(j)=0$ for all $j\in \mathrm{supp}(f_k)$, it follows that  
		\[\begin{aligned}
			\sum_{i\in G}\vert f_k(i)-T(x)(i)\vert ^p&=\sum_{i\in G\cap \mathrm{supp}(f_k)}\vert f_k(i)\vert^p+\sum_{i\in G\minus \mathrm{supp}(f_k)}\vert T(x)(i)\vert^p\\&\leqslant \sum_{i\in  \mathrm{supp}(f_k)}\vert f_k(i)\vert^p-\vert f_k(i_0)\vert^p+\sum_{i\in G\minus \mathrm{supp}(f_k)}\vert T(x)(i)\vert^p <2,
		\end{aligned}\]
		which is a contradiction. Hence, $\mathrm{supp}(f_k)\subset G$ and 
		\[
		\sum_{i\in\mathrm{supp}(f_k)}\vert f_k(i)\vert^p+\sum_{i\in G\minus \mathrm{supp}(f_k)}\vert T(x)(i)\vert^p=2.
		\]
		So, \[\sum_{i\in G\minus \mathrm{supp}(f_k)}\vert T(x)(i)\vert^p=1.\]
		We denote $G\minus \mathrm{supp}(f_k) $ by $F$. Note that $F\in \Sc_\alpha$ since $\Sc_\alpha$ is hereditary and $G\in \Sc_\alpha$. We claim that there exists $i_1\in \{1,\dots, r\}$ such that $k_{i_1}\notin F$. Indeed, assume that $\{k_1,\dots, k_r\}\subset F$. Since $G=\mathrm{supp}(f_k)\cup F\in \Sc_\alpha$, $k_0,m_0\in \mathrm{supp}(f_k)$ and $\Sc_\alpha$ is hereditary, it follows that $\{k_0,m_0,k_1,\dots,k_r\}\in \Sc_\alpha$, which contradicts the fact the $\{k_0,k_1,\dots,k_r\}\in \Sc_\alpha^{MAX}$. This proves our claim. \\
		%Since there exists $i_1\in \{1,\dots, r\}$ such that $k_{i_1}\notin F$ and since $T(x)(j)=0$ for all $j\geqslant j_0$ such that $j\notin \{k_1,\dots,k_r\}$, it follows that $\min F<k_0$ because $\sum_{i\in F}\vert T(x)(i)\vert^p=1$ and $\sum{i\in \{1,\dots,r\}\minus\{i_1\}}\vert T(x)(k_i)\vert^p<1$.
		Since $G=F\cup\mathrm{supp}(f_k)\in \Sc_\alpha$, $k_0\in \mathrm{supp}(f_k)$ (hence $k_0\notin F$) and $\Sc_\alpha$ is hereditary, it follows that $F\cup\{k_0\}\in \Sc_\alpha$. This implies that $F\cup\{k_{i_1}\}\in \Sc_\alpha$ since $\Sc_\alpha$ is spreading. Hence, 
		\[\begin{aligned}\Vert T(x)\Vert^p&\geqslant \sum_{i\in F\cup\{k_{i_1}\}}\vert T(x)(i)\vert^p\\&=\sum_{i\in F}\vert T(x)(i)\vert^p+\vert T(x)(k_{i_1})\vert^p\\&=1+\vert a_{i_1}\vert^p>1,
		\end{aligned} \]
		which is the required contradiction. So, we have shown that $k\leqslant \mathrm{supp}(f_k)$. In the same way, we show that $k \leqslant \mathrm{supp}(d_k)$. 
	\end{proof}
	We conclude this subsection by proving Proposition~\ref{main3}.
	\begin{lemma}\label{principal}
		For any $k\in \N$, there exists a sign $\theta_k$ such that $f_k=\theta_k e_k$ and $d_k=\theta_k e_k$.
	\end{lemma}
	\begin{proof}
		This is proved for $k=1$ in Lemma~\ref{l4}. Note that $f_1=\theta_1 e_1$ implies $d_1=\theta_1 e_1$ by Lemma~\ref{l2}.\\
		Assume now that $k \geqslant 2$. By Lemma~\ref{l12}, $k\leqslant \mathrm{supp}(f_k)$. Towards a contradiction, assume that there exists $m>k$ such that $f_k(m)\neq 0$. Then, by Lemma~\ref{l8}, we have $d_m(k)\neq 0$ which contradicts the fact that $m\leqslant \mathrm{supp}(d_m)$ (by Lemma~\ref{l12}) since $k<m$. So, $\mathrm{supp}(f_k)=\{k\}$ which implies that there exists a sign $\theta_k$ such that $f_k=\theta_k e_k$, and so $d_k=\theta_k e_k$ by Lemma~\ref{l6}.
	\end{proof}
	\subsection{The case $p=2$}\label{sec3}
	In this subsection, we assume that  $p=2$ and  $1 \leqslant\alpha<\omega_1$ without any additional assumption on the approximating sequence $(\alpha_n)$ for $\alpha$ in the limit case, \textit{i.e.} without assuming that the transfinite Schreier sequence $(\Sc_\alpha)_{\alpha<\omega_1}$ is good. As in the previous subsection, our aim is to prove Proposition~\ref{main3}.
	
	\begin{lemma}\label{Fact1}
		Let $i \in \N\minus \{1\}$ and $x \in \Sp$. If $\Vert e_i+x\Vert >1 $ and $\Vert e_i-x\Vert >1$, then $\Vert e_i+x\Vert ^2+\Vert e_i-x\Vert^2\leqslant 4$. 
	\end{lemma}
	\begin{proof}
		Assume that  $\Vert e_i+x\Vert >1 $ and $\Vert e_i-x\Vert >1$. Choose $F,G \in \Sc_\alpha$ such that 
		\[\Vert e_i+x\Vert^2=\sum_{k\in F}\vert e_i(k)+x(k)\vert^2 \]
		and 
		\[\Vert e_i-x\Vert^2=\sum_{k\in G}\vert e_i(k)-x(k)\vert^2. \]
		If $i \notin F $ then $\Vert e_i+x\Vert \leqslant 1$ which is a contradiction. So, $i \in F$ and similarly $i \in G$. Hence, 
		\[\begin{aligned}
			\Vert e_i+x\Vert^2+ \Vert e_i-x\Vert ^2&= \sum_{k\in F\minus \{i\}}\vert x(k)\vert^2+ \vert 1+x(i)\vert ^2+ \sum_{k\in G\minus \{i\}}\vert x(k)\vert^2+ \vert 1-x(i)\vert ^2\\&=\sum_{k\in F\minus \{i\}}\vert x(k)\vert^2+\sum_{k\in G\minus \{i\}}\vert x(k)\vert^2+ 2+2\vert x(i)\vert ^2\\&=\sum_{k\in F}\vert x(k)\vert ^2+\sum_{k\in G}\vert x(k)\vert ^2+2\leqslant4,
		\end{aligned}\]
		since $\sum_{k\in F}\vert x(k)\vert ^2\leqslant1$ and $\sum_{k\in G}\vert x(k)\vert ^2\leqslant1$ because $\Vert x \Vert =1$.
	\end{proof}
	\begin{lemma}\label{Fact2}
		Let $i \in \N\minus \{1\}$. Then, 
		\begin{align}\label{C1}f_i=a e_{j_1}+b e_{j_2} \qquad{\text{ where}} \quad j_1\neq j_2 \in\N\minus\{1\} \quad{\text{and}} \quad\vert a\vert=\vert b\vert=\sqrt2/2,\end{align}
		or
		\begin{align}\label{C2}f_i=\pm e_{j} \qquad{\text{where }} \quad j\in\N\minus\{1\}.\end{align}
	\end{lemma}
	\begin{proof}
		We have $f_i= \sum_{k=1}^n v_k e_{i_k}$ where $v_k\neq0$ for all $k\in \{1,\dots,n\}$, $i_1<\dots< i_n$ and $\{i_1,\dots, i_n\}\in \Sc_\alpha\minus \Sc_\alpha^{MAX}$ by Lemma~\ref{l7}. Towards a contradiction, assume that $f_i$ does not satisfy (\ref{C1}) or (\ref{C2}). Then, $n\geqslant2$ and there exists $l\in \{1,\dots ,n\}$ such that $\vert v_{l}\vert <\sqrt2/2$ since $\sum_{k=1}^n v_k^2=1$. Let us fix such an $l$ and $\varepsilon>0$ such that 
		\begin{align}\label{C3}\varepsilon<\frac{1-2v_{l}^2}{2\sum\limits_{k=1 \atop k\neq l}^n\vert v_k\vert }\end{align}
		and 
		\begin{align}\label{C4}
			\varepsilon^2<\frac{1-v_l^2}{n-1}\cdot
		\end{align}
		Moreover, we fix $m\geqslant 1$ such that $F:=\{i_1,\dots, i_n,i_n+1, \dots, i_n+m\}\in \Sc_\alpha^{MAX}$ (this is possible by Fact~\ref{max}), and let $\varepsilon_k$ be the sign of $v_k$ for all $k\in \{1, \dots, n\}$. Now, let us define $y\in X_{\Sc_\alpha,2}$ as follows: 
		\[y=v_le_{i_l}+\sum\limits_{k=1\atop k\neq l}^n \varepsilon_k \varepsilon e_{i_k}+\sum\limits_{k=1}^m \gamma e_{i_n+k}+\delta e_{i_n+m+1},\]
		where $\gamma >0 $ is such that 
		\begin{align}\label{C5}v_l^2+(n-1)\varepsilon^2+m\gamma^2=1,\end{align}
		and $\delta =\min\{\vert v_l\vert ,\varepsilon,\gamma\}$. Note that one can indeed find  such a $\gamma>0$ since $v_l^2+(n-1)\varepsilon^2<1$ by (\ref{C4}). Since $F \in \Sc_\alpha^{MAX}$ and $\delta =\min\{\vert v_l\vert ,\varepsilon,\gamma\}$, it is easy to check that $\Vert y \Vert=1$ by~(\ref{C5}).
		
		Now, we will show that $\Vert f_i+y\Vert>1$, $\Vert f_i-y\Vert>1$ and $\Vert f_i+y\Vert^2+\Vert f_i-y\Vert^2>4$. This will contradict Lemma~\ref{Fact1}, by Lemma~\ref{l6} and since $T$ is a surjective isometry.
		
		Since $F \in \Sc_\alpha$,  we have 
		\begin{align}\label{R1}\Vert f_i+y\Vert ^2&\geqslant \sum_{k\in F}\vert f_i(k)+y(k)\vert^2 \nonumber\\&=(v_l+v_l)^2+\sum\limits_{k=1\atop k\neq l}^n (v_k+\varepsilon_k\varepsilon)^2+m\gamma^2\nonumber\\&=\sum\limits_{k=1}^n v_k^2 + v_l^2+(n-1)\varepsilon^2+m\gamma^2+ 2v_l^2+2\varepsilon\sum\limits_{k=1\atop k\neq l}^n\vert v_k\vert\nonumber\\&=1+1+2v_l^2+2\varepsilon\sum\limits_{k=1\atop k\neq l}^n\vert v_k\vert>2,
		\end{align}
		and since  $(F\minus \{i_l\})\cup\{i_n+m+1\}\in \Sc_\alpha$, we have 
		\begin{align}\label{R2}
			\Vert f_i-y\Vert^2&= \sum\limits_{k=1\atop k\neq l}^n (v_k+\varepsilon_k\varepsilon)^2+m\gamma^2+\delta^2\nonumber\\&= \sum\limits_{k=1\atop k\neq l}^n v_k^2+(n-1)\varepsilon^2-2\varepsilon\sum\limits_{k=1\atop k\neq l}^n\vert v_k\vert+m\gamma^2+\delta^2\nonumber\\&=1-v_l^2+(n-1)\varepsilon^2-2\varepsilon\sum\limits_{k=1\atop k\neq l}^n\vert v_k\vert+m\gamma^2+\delta^2\\&>v_l^2+(n-1)\varepsilon^2+m\gamma^2+\delta^2=1+\delta^2>1,\nonumber
		\end{align}
		because $1-2\varepsilon\sum\limits_{k=1\atop k\neq l}^n\vert v_k\vert>2v_l^2$ by (\ref{C3}). 
		
		Therefore, by (\ref{R1}) and (\ref{R2}), we have 
		\[
		\Vert f_i+y\Vert^2+\Vert f_i-y\Vert^2\geqslant 3+v_l^2+(n-1)\varepsilon^2+m\gamma^2+\delta^2=4+\delta^2>4,
		\]
		which is the required contradiction.
	\end{proof}
	
	\begin{lemma}\label{Fact4}
		Let $i>j\in \N\minus\{1\}$. There exists $x\in \Sp$ such that $\Vert x+e_i\Vert^2=\Vert x-e_i\Vert^2=2 $, $\Vert x+e_j\Vert^2<2$ and $\Vert x-e_j\Vert^2<2$.
	\end{lemma}
	\begin{proof}
		By Fact~\ref{max}, there exists $m\geqslant 2$ such that $F:=\{i,i+2,\dots, i+m\}\in \Sc_\alpha^{MAX}$. Then $\{j, i+2,\dots, i+m\}\notin \Sc_\alpha$, since otherwise, we get by Fact~\ref{Fact3} that $\{i,i+1,i+2,\dots, i+m\}\in \Sc_\alpha$ which contradicts the fact that $F\in \Sc_\alpha^{MAX}$. Now, we take 
		\[x:=\sum_{k=2}^m a_ke_{i+k} \qquad{\rm where}\qquad\sum_{k=2}^m\vert a_k\vert^2=1 \quad{\rm and}\quad a_k\neq0 \quad{\text{for all}} \quad k\in \{2,\dots, m\}.\]
		It is easy to check that $x$ has the required properties.
	\end{proof}
	\begin{lemma}\label{Fact5}
		For any  $k\in \N\minus \{1\}$, there exists $j\in \N\minus \{1\}$ such that $d_k=\pm e_j$. 
	\end{lemma}
	\begin{proof}
		Let $k\in \N\minus \{1\}$. Towards a contradiction, assume that there exist $i\neq j \in \N\minus \{1\}$ such that $d_k(i)\neq 0$ and $d_k(j)\neq 0$. Then, by Lemma~\ref{l8},  $k \in \mathrm{supp}(f_i)\cap \mathrm{supp}(f_j)$, and so, by Lemma~\ref{Fact2}, $\#(\mathrm{supp}(f_i)\cup\mathrm{supp}(f_j))\leqslant 3$. Now, we will show that $\mathrm{supp}(f_i)\cup\mathrm{supp}(f_j )\in \Sc_\alpha$. 
		
		If $\alpha\geqslant 2$, then we are done since by the definition of $\Sc_\alpha$, every subset of $\N$ with cardinality at most $3$ belong to $\Sc_\alpha$. 
		
		If $\alpha=1$, then by the definition of $S_1$, any subset of $\N$ with cardinality at most $3$ and such that its minimum is greater than or equal to $3$ belongs to $\Sc_1$. So, it remains to show that $\min (\mathrm{supp}(f_i)\cup\mathrm{supp}(f_j))\neq 2$. Towards a contradiction, assume for example that $2\in \mathrm{supp} (f_i)$. By Lemma~\ref{l7} (and since $\alpha=1$), $\mathrm{supp}( f_i)=\{2\}$. Hence, $k=2$ and so $2\in \mathrm{supp} (f_j) $. By Lemma~\ref{l7} again, $\mathrm{supp} (f_j)=\{2\}$. Therefore, by lemma~\ref{l6}, $\mathrm{supp}(T(e_i))=\mathrm{supp}(T(e_j))=\mathrm{supp}(T(-e_i))=\mathrm{supp}(T(-e_j))=\{2\}$ which contradicts the injectivity of $T$. 
		
		Since $\mathrm{supp}(f_i)\cup\mathrm{supp}(f_j) \in \Sc_\alpha$, it follows that $\Vert f_i\pm f_j\Vert^2 _{\ell_2}=\Vert f_i\pm f_j\Vert^2=\Vert e_i \pm e_j\Vert ^2=2 $. So, $f_i$ and $f_j$ are orthogonal in $\ell_2$ (and so $k$ cannot be the unique point of intersection of  $\mathrm{supp} (f_i)$ and $\mathrm{supp} (f_j)$). So there exists $l\neq k \in \N\minus \{1\}$ such that $\mathrm{supp}(f_i)=\mathrm{supp}(f_j)=\{k,l\}$. Without loss of generality, assume that $i>j$ and that 
		\[ f_i=\frac{\sqrt{2}}{2}e_k+\frac{\sqrt{2}}{2}e_l\]
		and \[f_j=\frac{\sqrt{2}}{2}e_k-\frac{\sqrt{2}}{2}e_l\]
		by Lemma~\ref{Fact2}. 
		
		Now, we will show that for any $y \in \Sp$, if $\Vert y+f_i\Vert ^2=\Vert y-f_i\Vert ^2=2$, then $\Vert y+f_j\Vert ^2\geqslant2$ or $\Vert y-f_j\Vert ^2\geqslant2$. This will contradict Lemma~\ref{Fact4}, by Lemma~\ref{l6} and since $T$ is a surjective isometry.
		
		Let $y \in \Sp$ be such that $\Vert y+f_i\Vert ^2=\Vert y-f_i\Vert^2=2$. Then, there exist $F,G \in \Sc_\alpha$ such that \begin{align}\label{F} \Vert y+f_i\Vert^2=\sum_{r\in F}\vert y(r)+f_i(r)\vert ^2 =2 \qquad{\rm where } \quad y(r)+f_i(r)\neq 0 \text{ for all }r\in F,\end{align}
		and  \begin{align}\label{G}
			\Vert y-f_i\Vert^2=\sum_{r\in G}\vert y(r)-f_i(r)\vert ^2 =2 \qquad{\rm where } \quad y(r)-f_i(r)\neq 0 \text{ for all }r\in G.\end{align} 
		Note that $F$ contains $k$ or $l$ since otherwise, we get $\Vert y+f_i\Vert^2=\sum_{r\in F}\vert y(r)\vert ^2\leqslant 1$, which is a contradiction. Similarly, we have that $G$ contains $k$ or $l$.
		
		-- If $F$ contains $l$ and not $k$, then 
		\[ \Vert y-f_j\Vert^2\geqslant\sum_{r\in F}\vert y(r)-f_j(r)\vert^2=\sum_{r\in F}\vert y(r)+f_i(r)\vert^2=2.\]
		
		-- In the same way, we show that $\Vert y+f_j\Vert^2\geqslant2$ or  $\Vert y-f_j\Vert^2\geqslant2$ if $F$ contains $k$ and not $l$, or $G$ contains $l$ and not $k$, or $G$ contains $k$ and not $l$.
		
		-- If $\{k,l\}\subset F\cap G$ and $\mathrm{supp}(y)\cap \{k,l\}=\emptyset$, then it is easy to see that $\Vert y+f_j\Vert^2=\Vert y-f_j\Vert^2=2$. 
		
		-- If $\{k,l\}\subset F\cap G$ and $\mathrm{supp}(y)$ contains $k$ and not $l$, then 
		\[ \Vert y+f_j\Vert ^2\geqslant\sum_{r\in F}\vert y(r)+f_j(r)\vert^2=\sum_{r\in F}\vert y(r)+f_i(r)\vert^2=2,\]
		and 
		\[\Vert y-f_j\Vert ^2\geqslant\sum_{r\in G}\vert y(r)-f_j(r)\vert^2= \sum_{r\in G}\vert y(r)-f_i(r)\vert^2=2.\]
		
		-- In the same way, we show that $\Vert y+f_j\Vert^2\geqslant2$ and $\Vert y-f_j\Vert^2\geqslant2$ if  $\{k,l\}\subset F\cap G$ and $\mathrm{supp}(y)$ contains $l$ and not $k$.
		
		-- If $\{k,l\}\subset F\cap G\cap \mathrm{supp}(y)$, then, by (\ref{F}) and (\ref{G}), we have 
		\begin{align}\label{F1}
			\biggl(y(k)+\frac{\sqrt 2}{2}\biggr)^2+\biggl(y(l)+ \frac{\sqrt 2}{2}\biggr)^2+\sum_{r\in F\minus \{k,l\}}\vert y(r)\vert^2=2, 
		\end{align}
		and  \begin{align}\label{G1}
			\biggl(y(k)-\frac{\sqrt 2}{2}\biggr)^2+\biggl(y(l)-\frac{\sqrt 2}{2}\biggr)^2+\sum_{r\in G\minus \{k,l\}}\vert y(r)\vert^2=2. 
		\end{align}
		Moreover, we have 
		\begin{align}\label{F2}
			\Vert y+f_j\Vert^2\geqslant \sum_{r\in F}\vert y(r)+f_j(r)\vert^2=   \biggl(y(k)+\frac{\sqrt 2}{2}\biggr)^2+\biggl(y(l)- \frac{\sqrt 2}{2}\biggr)^2+\sum_{r\in F\minus \{k,l\}}\vert y(r)\vert^2,
		\end{align}
		and 
		\begin{align}\label{G2}
			\Vert y-f_j\Vert^2\geqslant \sum_{r\in G}\vert y(r)-f_j(r)\vert^2=   \biggl(y(k)-\frac{\sqrt 2}{2}\biggr)^2+\biggl(y(l)+\frac{\sqrt 2}{2}\biggr)^2+\sum_{r\in G\minus \{k,l\}}\vert y(r)\vert^2.
		\end{align}
		Hence, (\ref{F1}), (\ref{G1}), (\ref{F2}) and (\ref{G2}) imply that 
		\[\Vert y+f_j\Vert^2+\Vert y-f_j\Vert ^2\geqslant4,\]
		and so $\Vert y+f_j\Vert^2\geqslant2$ or $\Vert y-f_j\Vert^2\geqslant2$.

		\smallskip\smallskip
		So we have shown that for any $y \in \Sp$ such that $\Vert y+f_i\Vert ^2=\Vert y-f_i\Vert ^2=2$, we have $\Vert y+f_j\Vert ^2\geqslant2$ or $\Vert y-f_j\Vert ^2\geqslant2$ which is the required contradiction. Hence, $d_k=\pm e_j$ for some $j\in \N\minus \{1\}$, because $d_k(1)=0$ by Lemma \ref{l5}.
	\end{proof}
	
	\smallskip\smallskip
	Now, we prove Proposition~\ref{main3}. To do this, we could use Lemma~\ref{Fact5} and Lemma~\ref{equality} which is valid for $p=2$ as well.  However, we will provide a direct proof that does not require $(\Sc_\alpha)_{\alpha<\omega_1}$ to be good.
	\begin{lemma}\label{Fact6}
		For any $k\in \N$, there exists a sign $\theta_k$ such that $d_k=\theta_k e_k$ and $f_k=\theta_k e_k$.
	\end{lemma}
	\begin{proof}
		This is proved for $k=1$ in Lemma~\ref{l4}. Note that $f_1=\theta_1 e_1$ implies $d_1=\theta_1 e_1$ by Lemma~\ref{l2}. So let us prove the result for any $k \in \N\minus \{1\}$. For this, we define the map $\sigma: \N\minus\{1\}\to \N\minus \{1\}$ such that $d_k=\pm e_{\sigma(k)}$ for all $k\in\N\minus \{1\}$. The map $\sigma$ is well defined by Lemma~\ref{Fact5}, and it is surjective. Indeed, towards a contradiction, assume that there exists $n\in \N\minus\{1\}$ such that $d_k\neq\pm e_n$ for all $k\in \N\minus \{1\}$. This means by Lemma~\ref{Fact5} that $n\notin \mathrm{supp}(d_k)$ for all $k\in \N\minus \{1\}$. Hence, $f_n(k)=0$ for all $k\in \N\minus \{1\}$ by Lemma~\ref{l8}. This implies that $f_n=0$ since $f_n(1)=0 $ by Lemma~\ref{l5}, which is a contradiction since $f_n \in\Sp$.
		
		Now, we will show that $\sigma$ is strictly increasing.  Let $i > j\in \N\minus \{1\}$. Assume that $\sigma(i)\leqslant \sigma(j)$. As in the proof of Lemma~\ref{Fact5}, we want to get a contradiction with Lemma~\ref{Fact4}. So let $y\in \Sp$ be such that $\Vert y+d_i\Vert^2=\Vert y-d_i\Vert^2=2$. Then, there exist $F,G\in \Sc_\alpha$ such that 
		\[\sum_{k\in F}\vert y(k)+e_{\sigma(i)}(k)\vert^2 =2\]
		and 
		\[\sum_{k\in G}\vert y(k)-e_{\sigma(i)}(k)\vert^2 =2.\]
		If $\sigma(i)\notin F \cap G$, then $\sum_{k\in F}\vert y(k)+e_{\sigma(i)}(k)\vert^2\leqslant 1$ or $\sum_{k\in G}\vert y(k)-e_{\sigma(i)}(k)\vert^2\leqslant1$ which is a contradiction. Hence $\sigma(i)\in F\cap G$. \\
		If $\sigma(i)\in \mathrm{supp}(y) $, then $\vert y(\sigma(i))+1\vert <1$ or $\vert y(\sigma(i))-1\vert <1$ and hence it is easy to check that $\sum_{k\in F}\vert y(k)+e_{\sigma(i)}(k)\vert^2<2$ or $\sum_{k\in G}\vert y(k)-e_{\sigma(i)}(k)\vert^2<2$ which is a contradiction. Hence, $\sigma(i)\notin \mathrm{supp}(y)$, and therefore,\begin{align}\label{F3}
			\sum_{k\in F\minus \{\sigma(i)\} }\vert y(k)\vert^2=1. 
		\end{align} 
		Since $\Sc_\alpha$ is spreading, $F\in S_\alpha$ and we are assuming that $\sigma(j)\geqslant\sigma(i)$, it follows that 
		\begin{align}\label{F4}
			(F\minus \{\sigma(i)\})\cup \sigma(j)\in \Sc_\alpha.
		\end{align}
		First, assume that $\sigma(j) \in \mathrm{supp}(y)$. Then, (\ref{F3}) and (\ref{F4}) imply that $\sigma(j)\in F\minus \{\sigma(i)\}$, since otherwise we get $\Vert y\Vert>1$. Hence, by (\ref{F3}), we have $\Vert y+d_j\Vert ^2>2$ or $\Vert y-d_j\Vert^2>2$, because 
	\[	\begin{aligned}
			\vert y(\sigma(j))+1\vert^2>1+\vert y(\sigma(j))\vert ^2 \qquad{\rm or } \qquad  \vert y(\sigma(j))-1\vert^2>1+\vert y(\sigma(j))\vert ^2.
		\end{aligned}\]
		
		Now, we  assume that $\sigma(j) \notin \mathrm{supp}(y)$. Then, by (\ref{F3}) and (\ref{F4}), it is easy to check that $\Vert y+d_j\Vert^2=\Vert y-d_j\Vert^2=2$.
		
		So, we have shown that for any $y\in \Sp$ such that $\Vert y+d_i\Vert^2=\Vert y-d_i\Vert^2=2$, we have $\Vert y+d_j\Vert ^2\geqslant2$ or $\Vert y-d_j\Vert^2\geqslant2$, which contradicts Lemma~\ref{Fact4}. Hence, $\sigma(i)>\sigma(j)$ and $\sigma$ is strictly increasing. 
		
		Since $\sigma : \N\minus \{1\}\to \N\minus \{1\}$ is  strictly increasing and surjective, it follows that $\sigma$ is the identity map. Therefore, for any $k\in \N\minus \{1\}$, there exists a sign $\theta_k$ such that $d_k=\theta_k e_k$, and so  $f_k=\theta_k e_k$ by Lemma~\ref{l6}.
	\end{proof}
	\subsection{Proof of Theorem~\ref{main2} for $p\in (1,\infty)$}\label{sec4} Now, we show Theorem~\ref{main2} for $p\in (1,\infty)$, using Proposition~\ref{main3}, which has been proved in subsections~\ref{sec2} (for $p\in (1,\infty)\minus\{2\} $) and~\ref{sec3} (for $p=2$). Specifically, we use the fact that there exists a sequence of signs $(\theta_i)_{i\in \N}$ such that $T(e_i)=\theta_i e_i$ (and $T^{-1}(e_i)=\theta_i e_i$) for all $i\in \N$. 
	\begin{lemma}\label{zero}
		Let $x\in \Sp$ and $i\in \N$. If $x(i)=0$, then $T(x)(i)=0$ and  $T^{-1}(x)(i)=0$.
	\end{lemma}
	\begin{proof}
		It is enough to show that $T(x)(i)=0$. If $i=1$, then this is true by Lemma~\ref{l5}. So, let us assume that $i\geqslant 2$. If $x=\pm e_1$, then $T(x)=\pm e_1$ by Lemma \ref{l4}, and so $T(x)(i)=0$. \\Now, suppose that $x\neq \pm e_1$. Hence, there exists $j\geqslant 2$ such that $x(j)\neq 0$. Since $\{i,j\}\in\Sc_\alpha$ and $x(i)=0$, it follows that 
		\begin{align}\label{machin}\Vert x\pm e_i\Vert ^p\geqslant \vert x(j)\vert^p+1>1.\end{align} Let $G\in \Sc_\alpha$ be such that $\Vert x+e_i\Vert^p=\sum_{k\in G}\vert x(k)+e_i(k)\vert^p$. If $i\notin G$, then $\Vert x+e_i\Vert\leqslant 1$ which contradicts (\ref{machin}). Hence, $i\in G$ and \[\Vert x+e_i\Vert^p=1+\sum_{k\in G\minus \{i\}}\vert x(k)\vert^p,\] because $x(i)=0$. Since $\vert (x+e_i)(k)\vert=\vert (x-e_i)(k)\vert$ for any $k\in \N$ (because $x(i)=0$), it follows that \[\Vert x-e_i\Vert ^p=\Vert x+e_i\Vert^p=1+\sum_{k\in G\minus \{i\}}\vert x(k)\vert^p.\]
		Therefore, by Proposition~\ref{main3} and Lemma~\ref{l6}, we have 
		\[\Vert T(x)-\theta_ie_i\Vert^p=\Vert T(x)+\theta_ie_i\Vert^p=1+\sum_{k\in G\minus \{i\}}\vert x(k)\vert^p.\]
		Towards a contradiction, assume that $T(x)(i)\neq 0$. Then $\mathrm{sgn}(T(x)(i))=\theta_i$ or $\mathrm{sgn}(T(x)(i))=-\theta_i$. Suppose that $\mathrm{sgn}(T(x)(i))=\theta_i$ and let $F\in \Sc_\alpha$ be such that $\Vert T(x)-\theta_ie_i\Vert^p=\sum_{k\in F}\vert T(x)(k)-\theta_ie_i(k)\vert^p$.\\
		Recall that $\Vert T(x)-\theta_ie_i\Vert=\Vert x-e_i\Vert>1$ by (\ref{machin}). If $i\notin F$, then $\Vert T(x)-\theta_i e_i\Vert\leqslant1$ which is a contradiction. Hence, $i\in F$ and
		\[\begin{aligned}\Vert  T(x)-\theta_ie_i\Vert^p&=\vert T(x)(i)-\theta_i\vert^p+\sum_{k\in F\minus\{i\}}\vert T(x)(k)\vert^p\\&<1+ \sum_{k\in F\minus\{i\}}\vert T(x)(k)\vert^p,\end{aligned}\]
		because we are assuming that $\mathrm{sgn}(T(x)(i))=\theta_i$. However, since $\Vert T(x)-\theta_ie_i\Vert^p=1+ \sum_{k\in G\minus\{i\}}\vert x(k)\vert^p$, it follows that 
		\[\sum_{k\in G\minus\{i\}}\vert x(k)\vert^p < \sum_{k\in F\minus\{i\}}\vert T(x)(k)\vert^p. \]
		Therefore, 
		\[\begin{aligned}
			\Vert T(x)+\theta_ie_i\Vert^p&\geqslant\sum_{k\in F}\vert T(x)(k)+\theta_ie_i(k)\vert^p\\&=\vert T(x)(i)+\theta_i\vert^p+\sum_{k\in F\minus \{i\}}\vert T(x)(k)\vert^p\\&>1+\sum_{k\in F\minus \{i\}}\vert T(x)(k)\vert^p \\&> 1+ \sum_{k\in G\minus\{i\}}\vert x(k)\vert^p \\&=\Vert x+e_i\Vert^p,
		\end{aligned}\]
		which is a contradiction. In the same way, we get a contradiction if we suppose that $\mathrm{sgn}(T(x)(i))=-\theta_i$. Hence, $T(x)(i)=0$.
	\end{proof}
	% In the following, for any $x\in \Sp$, we set 
	% \[ \mathcal{A}_x^1:=\biggl\{F\in \Sc_\alpha;\;\sum_{i\in F}\vert x(i)\vert^p=1\biggr\}.\]
	% Recall that for any $x\in \Sp$, we have $\mathcal{A}_x^1\neq \emptyset$ by ~\ref{normeatteinte}.
	\begin{lemma}\label{final}
		For any $x\in \Sp$ and $i\in \N$, we have $T(x)(i)=\theta_i x(i)$. 
	\end{lemma}
	\begin{proof}
		Let $x\in \Sp$ and $i\in \N$. If $x(i)=0$, then $T(x)(i)=0$ by Lemma~\ref{zero}. So, in the rest of the proof we will suppose that $x(i)\neq0$. We will show first that $\vert T(x)(i)\vert =\vert x(i)\vert$. If there exists $F \in \mathcal{A}_x^1:=\{G\in \Sc_\alpha;\;\sum_{k\in G}\vert x(k)\vert^p=1\}$ such that $i\in F$, then $\vert T(x)(i)\vert=\vert x(i)\vert $ by Lemma~\ref{imp} and since $T(e_i)=\theta_i e_i $ for all $i \in \N$. 
		
		Assume now that for any $F\in \mathcal{A}_x^1$, $i\notin F$. We  define  $y\in X_{\Sc_\alpha,p}$ as follows: $y(k)=x(k)$ for all $k\in \N\minus \{i\}$ and $y(i)=0$. We have $y \in \Sp$ because for any $F\in \mathcal{A}_x^1$, $i\notin F$ and so 
		\[\sum_{k\in F}\vert y(k)\vert ^p=\sum_{k\in F}\vert x(k)\vert^p=1,\]
		and for any $G\in \Sc_\alpha\minus \mathcal{A}_x^1 $,
		\[\sum_{k\in G}\vert y(k)\vert^p\leqslant \sum_{k\in G}\vert x(k)\vert^p< 1. \]
		Since $T(y)(i)=0$ by Lemma~\ref{zero}, it follows that 
		\begin{align}\label{P1}
			\vert x(i)\vert =\Vert x-y\Vert=\Vert T(x)-T(y)\Vert \geqslant\vert T(x)(i)-T(y)(i)\vert =\vert T(x)(i)\vert. 
		\end{align}
		Since $T(x)(i)\neq 0$ by Lemma \ref{zero} (because $x(i)\neq 0$), one can apply the same argument to $T^{-1}$ and $T(x)$ instead of $T$ and $x$ to obtain that 
		\[\vert T(x)(i)\vert \geqslant \vert T^{-1}(T(x))(i)\vert,\]
		\textit{i.e.}
		\begin{align}\label{P2}
			\vert T(x)(i)\vert \geqslant \vert x(i)\vert. 
		\end{align}
		So, (\ref{P1}) and (\ref{P2}) imply that $\vert T(x)(i)\vert =\vert x(i)\vert $. Now, we will show  that $T(x)(i)=\theta_i x(i)$, \textit{i.e.} $\mathrm{sgn}(T(x)(i))=\theta_i \mathrm{sgn}(x(i))$. We denote by $\varepsilon$ the sign of $x(i)$. Towards a contradiction, assume that $\mathrm{sgn}(T(x)(i))=-\varepsilon\theta_i$. 
		
		Since $T(e_i)=\theta_i e_i$, we have by Lemma~\ref{l2} and Lemma~\ref{l6} that \[
		\Vert T(x)+\varepsilon\theta_ie_i\Vert =\Vert x+\varepsilon e_i\Vert \geqslant\vert x(i)+\varepsilon\vert>1, \]
		because $\varepsilon=\mathrm{sgn}(x(i))$. So there exists $F\in \Sc_\alpha$ such that 
		\[\Vert T(x)+\varepsilon\theta_i e_i\Vert^p=\sum_{k\in F}\vert T(x)(k)+\varepsilon\theta_ie_i(k)\vert^p>1.\]
		If $i\notin F$, then $\Vert T(x)+\varepsilon\theta_i e_i\Vert\leqslant1$ which is a contradiction. Hence $i \in F$. 
		
		Since we are assuming that  $\mathrm{sgn}(T(x)(i))=-\varepsilon\theta_i$ and since we have shown that $\vert T(x)(k)\vert=\vert x(k)\vert$ for all $k\in \N$, it follows that
		\[\begin{aligned}
			\Vert T(x)+\varepsilon\theta_i e_i\Vert^p&=\sum_{k\in F}\vert T(x)(k)+\varepsilon\theta_ie_i(k)\vert^p\\&=\vert T(x)(i)+\varepsilon \theta_i \vert^p+\sum_{k\in F\minus\{i\} }\vert T(x)(k)\vert^p\\&<1+\sum_{k\in F\minus\{i\} }\vert T(x)(k)\vert^p\\&=1+\sum_{k\in F\minus\{i\} }\vert x(k)\vert^p\\&<\vert x(i)+\varepsilon\vert^p+ \sum_{k\in F\minus\{i\} }\vert x(k)\vert^p\\&=\sum_{k\in F}\vert x(k)+\varepsilon e_i(k)\vert^p\\&\leqslant\Vert x+\varepsilon e_i\Vert^p,
		\end{aligned}\]
		which is a contradiction because $\Vert T(x)+\varepsilon\theta_ie_i\Vert =\Vert x+\varepsilon e_i\Vert$.
	\end{proof}
	
	In conclusion, we have proved Theorem~\ref{main2} for $p\in (1,\infty)$, which  immediately implies  Theorem~\ref{main}, as explained in the introduction. 
	
	\section{The case $p=1$}\label{secp1}
	In this section, our aim is to prove Theorem~\ref{main2} for  $p=1$ and  $1 \leqslant\alpha<\omega_1$ without any additional assumption on the approximating sequence $(\alpha_n)$ for $\alpha$ in the limit case. So, let us fix a transfinite Schreier sequence $(\Sc_\alpha)_{\alpha<\omega_1}$, an ordinal $1\leqslant \alpha<\omega_1$ and a surjective isometry $T:\Sp\to \Sp$ where $\Sp$ is the unit sphere of $X_{\Sc_\alpha}$. For simplicity, we will write $\Vert\,\cdot\,\Vert$ instead of $\Vert \,\cdot\,\Vert_{\Sc_\alpha}$.

	\smallskip \smallskip For any $x\in \Sp$, we set 
	\[ \mathcal{A}_x^1:=\biggl\{G\in \Sc_\alpha;\;\sum_{i\in G}\vert x(i)\vert=1\biggr\},\]
	and we call $F\in \Sc_\alpha$ a \emph{$1$-set for $x$} if $\sum_{i\in F}\vert x(i)\vert =1$ and $x(i)\neq 0$ for every $i\in F$. Note that if $G\in  \mathcal{A}_x^1 $, then by heredity, there exists a $1$-set $F$ for $x$ such that $F\subseteq G$.

	\smallskip \smallskip
	We start with the following important fact.
	\begin{fact}\label{normeatteintep1}
		Let $x\in \Sp$.
		\begin{enumerate}[label={\rm \arabic*)}]
			\item  The set of all $1$-sets of $x$ is non-empty and finite.
			\item   There exists an $\varepsilon_x>0$ $($called the $\varepsilon$-gap for $x$$)$ such that for any $F\in \Sc_\alpha\minus \mathcal{A}_x^1 $, we have $\sum_{i\in F}\vert x(i)\vert <1-\varepsilon_x$.
		\end{enumerate}  
	\end{fact}
	\begin{proof}
		This is~\cite[Lemmas 2.4, 2.5]{BDHQ}.
	\end{proof}
	Now, we address a series of lemmas that lead us to prove Proposition~\ref{main3} first and then Theorem~\ref{main2}.
	
	\begin{lemma}\label{lemma1}
		Let $x,y \in \Sp$ and let $F\in \Sc_\alpha$ be such that  $\Vert x+y\Vert =\sum_{k\in F}\vert x(k)+y(k)\vert =2$. Then $F\in \mathcal{A}_x^1\cap \mathcal{A}_y^1$  and for any $k\in\mathrm{supp}(x)\cap \mathrm{supp}(y)\cap F$, we have $\mathrm{sgn}(x(k))=\mathrm{sgn}(y(k))$.    
	\end{lemma}
	\begin{proof}
		Towards a contradiction, assume that $F \notin \mathcal{A}_x^1 \cap \mathcal{A}_y^1 $. Then
		\[ \sum_{k\in F} \vert x(k)+y(k)\vert \leqslant  \sum_{k\in F} \vert x(k)\vert +  \sum_{k\in F} \vert y(k)\vert <2,\]
		which is a contradiction. 
		
		Now, let $i\in \mathrm{supp}(x)\cap \mathrm{supp}(y)\cap F$ and assume towards a contradiction that $\mathrm{sgn}(x(i))=-\mathrm{sgn}(y(i))$. Hence, $\vert x(i)+y(i)\vert < \vert x(i)\vert +\vert y(i)\vert $. Therefore, 
		\[\sum_{k\in F}\vert x(k)+y(k)\vert < \sum_{k\in F} \vert x(k)\vert +  \sum_{k\in F} \vert y(k) \vert =2,\]
		which is  a contradiction again.
	\end{proof}
	The proof of the next lemma is similar to that of~\cite[Lemma 2.2]{LLi}.
	\begin{lemma}\label{lemma2}
		Let $n \in \N\minus \{1\}$ and $y\in \Sp$.
		\begin{enumerate}[label={\rm \arabic*)}]
			\item  If $T(y)=-T(e_n)$ or $T^{-1}(y)=-T^{-1}(e_n)$, then $y(n)=-1$.% and $y(i)=0$ for all $i \in \N\minus \{1,n\}$.
			\item   If $T(y)=-T(-e_n)$ or $T^{-1}(y)=-T^{-1}(-e_n)$, then $y(n)=1$. %and $y(i)=0$ for all $i \in \N\minus \{1,n\}$.
		\end{enumerate}
	\end{lemma} 
	\begin{proof}
		Assume that $T(y)=-T(e_n)$. Towards a contradiction, assume that there exists $m\in \N\minus \{1,n\} $ such that $y(m)\neq 0$. 
		Let us define  \[z:= \frac{1}{2}e_n+\frac{\mathrm{sgn}(y(m))}{2}e_m.\]
		It is clear that $z\in \Sp$. Hence, there exists $u\in \Sp$ such that $T(u)=-T(z)$, and so 
		\[\Vert u-z\Vert =\Vert T(u)-T(z)\Vert =2. \]
		So,  there exists $F\in \Sc_\alpha$ such that $\sum_{k\in F} \vert u(k)-z(k)\vert =2$. Since $\{n,m\}$ is the only $1$-set  of $z$, it follows by Lemma~\ref{lemma1} that $\{n,m\} \subset F$, $F \in \mathcal{A}_u^1$  and that $u(n)=0$ or $u(n)<0$. Hence, 
		\[ \Vert u -e_n \Vert \geqslant \sum_{k\in F} \vert u(k)-e_n(k)\vert= \vert u(n)-1\vert + \sum_{k\in F\minus \{n\}}\vert u(k)\vert =1+\vert u(n)\vert + \sum_{k\in F\minus \{n\}}\vert u(k)\vert=2. \]
		So, 
		\[\Vert z-y\Vert = \Vert T(z)-T(y)\Vert = \Vert -T(u)+ T(e_n)\Vert =\Vert e_n-u\Vert =2.  \]
		As before, since $\{n,m\}$ is the only $1$-set  of $z$, we get by Lemma~\ref{lemma1} that $y(m)=0$ or $\mathrm{sgn}(y(m))=-\mathrm{sgn}(z(m))$. Since we are assuming that $y(m)\neq 0$ and since $z(m)=\frac{\mathrm{sgn}(y(m))}{2}$, it follows that $\mathrm{sgn}(y(m))=-\mathrm{sgn}(z(m))=-\mathrm{sgn}(y(m)) $, which is a contradiction. Therefore, $y(m)=0$ for any $m\in \N\minus \{1,n\}$. 
		
		Now, assume that $y(n)\neq -1$. Since $y(m)=0$ for any $m\in \N\minus \{1,n\}$, it follows that $\Vert y-e_n\Vert =\max\{\vert 1-y(n)\vert, \vert y(1)\vert\} <2$, which is a contradiction because $\Vert y-e_n\Vert=\Vert T(y)-T(e_n)\Vert =\Vert -T(e_n)-T(e_n)\Vert=2 $. So, $y(n)=-1$. 
		
		The other cases are treated in the same way. 
	\end{proof}
	\begin{lemma}\label{lemma3}
		We have $T(-e_1)=-T(e_1)$ and  $T^{-1}(-e_1)=-T^{-1}(e_1)$.
	\end{lemma}
	\begin{proof}
		Let $x\in \Sp$ be such that $T(x)=-T(e_1)$. We need to show that $x=-e_1$. Since $T(x)=-T(e_1)$, it follows that
		\[ \Vert x-e_1\Vert =\Vert T(x)-T(e_1)\Vert=2.  \]
		So, there exists $F\in \Sc_\alpha$ such that $\sum_{k\in F}\vert x(k)-e_1(k)\vert =2$. If $1\notin F$, then \[\sum_{k\in F}\vert x(k)-e_1(k)\vert =\sum_{k\in F}\vert x(k)\vert \leqslant 1,\] which is a contradiction. Hence $1\in F$ and so $F=\{1\}$ since $\{1\}\in \Sc_\alpha^{MAX}$. Therefore, $\vert 1-x(1)\vert =2$ which implies that $x(1)=-1$. 
		
		Towards a contradiction, assume that there exists $n\in \N\minus \{1\}$ such that $x(n)\neq 0$. Assume first that $x(n)> 0$ . Let $y \in \Sp$ be such that $T(y)=-T(e_n)$. By Lemma~\ref{lemma2}, we have $y(n)=-1$ and hence 
		\[\Vert  y- x\Vert \geqslant\vert y(n)-x(n)\vert =1+x(n)>1, \]
		which is a contradiction since 
		\[\Vert y-x\Vert =\Vert T(y)-T(x)\Vert =\Vert -T(e_n)+T(e_1)\Vert=\Vert e_1-e_n\Vert =1.\]
		In the same way, we get a contradiction when $x(n)<0$ (in this case, we let $y\in \Sp$ be such that $T(y)=-T(-e_n)$ and we use 2) of Lemma~\ref{lemma2}).
		
		Therefore, $x(1)=-1$ and $x(i)=0$ for all $i\in \N\minus \{1\}$, \textit{i.e.} $x=-e_1$. 
		
		Analogously, we show that $T^{-1}(-e_1)=-T^{-1}(e_1)$.
	\end{proof}
	\begin{lemma}\label{lemma4}
		There exists a sign $\theta_1$ such that $T(e_1)=\theta_1 e_1$ and $T^{-1}(e_1)=\theta_1 e_1$.     
	\end{lemma}
	\begin{proof}
		The proof is the same as that of Lemma~\ref{l4} since $T(-e_1)=-T(e_1)$ by Lemma~\ref{lemma3}, and since Lemma~\ref{l3} works as well for $p=1$.
	\end{proof}
	\begin{lemma}\label{lemma5}
		$T(A)=A$ where $A$ is the set of all  $x\in \Sp$ such that $1\in \mathrm{supp}(x)$.
	\end{lemma}
	\begin{proof}
		The proof is the same as that of Lemma~\ref{l5}.
	\end{proof}
	\begin{lemma}\label{lemma6}
		For any $n\in \N\minus \{1\}$, we have $T(-e_n)=-T(e_n)$ and $T^{-1}(-e_n)=-T^{-1}(e_n)$.
	\end{lemma}
	\begin{proof}
		Let $n\in \N\minus \{1\}$ and $y \in \Sp$ be such that $T(y)=-T(e_n)$. By Lemma ~\ref{lemma2}, we have $y(n)=-1$ and hence $y(i)=0$ for all $i\in \N\minus\{1,n\}$ since otherwise we get $\Vert y\Vert >1$ because $\{i,n\}\in \Sc_\alpha$. Moreover, since $T(A)=A$ by Lemma~\ref{lemma5}, and since $e_n\in \Sp\minus A$, it follows that $T(e_n)\in \Sp\minus A$ and so $y=T^{-1}(-T(e_n))\in \Sp\minus A$. Hence, $y(1)=0$ and so $y=-e_n$. In the same way one shows that $T^{-1}(-e_n)=-T^{-1}(e_n)$.
	\end{proof}
	
	\begin{lemma}\label{lemma7}
		Let $n\in \N\minus \{1\}$ and $x\notin\{ e_1, -e_1\}$. Then 
		\[ \Vert e_n+x\Vert +\Vert e_n-x\Vert =2 \; \Longleftrightarrow \; x=\pm e_n. \]
	\end{lemma}
	\begin{proof}
		The reverse implication is obvious. For the forward implication. Since $x\in \Sp$ and $x\notin\{ e_1, -e_1\}$, it follows that there exists $m\in \N\minus\{1\} $ such that $x(m)\neq0$. 
		
		Assume first towards a contradiction that $x(n)=0$. Since $m\neq n \in \N\minus \{1\}$, it follows that $\{m,n\}\in \Sc_\alpha$ and hence 
		\[\Vert x+e_n\Vert + \Vert x-e_n\Vert\geqslant 1+\vert x(m)\vert + 1+\vert x(m)\vert>2,\]
		which is a contradiction. Hence, $x(n)\neq 0$. We denote by $\varepsilon$  the sign of $x(n)$.
		
		\smallskip
		Assume now  that $\vert x(1)\vert =1$. We have 
		\[\Vert x+ \varepsilon e_n\Vert \geqslant 1+ \vert x(n)\vert >1,\]
		and $\Vert x- \varepsilon e_n\Vert \geqslant1$ because we are assuming that $\vert x(1)\vert =1$. So, $\Vert x+ \varepsilon e_n\Vert+ \Vert x- \varepsilon e_n\Vert >2$ which is a contradiction. Hence, $\vert x(1)\vert \neq 1$. 
		
		\smallskip
		Assume that $\vert x(n)\vert\neq 1$. Then, since we have shown that $\vert x(1)\vert \neq 1$, one can find $m\in \N\minus \{1, n\}$ such that $x(m)\neq 0$, because otherwise we get that $\Vert x\Vert <1$. Since $\{m,n\}\in \Sc_\alpha$, it follows that 
		\[\Vert x+ \varepsilon e_n\Vert+ \Vert x- \varepsilon e_n\Vert\geqslant 1+\vert x(n)\vert +\vert x(m)\vert +1-\vert x(n)\vert +\vert x(m)\vert >2, \]
		which is a contradiction. Hence $\vert x(n)\vert =1$ and so $x(i)=0$ for all $i\in \N\minus \{1,n\}$. 
		
		\smallskip
		It remains to show that $x(1)=0$. Towards a contradiction, assume that $x(1)\neq 0$. Since $x(n)=\varepsilon$, we have $\Vert x+\varepsilon e_n\Vert = 2$. Moreover, $\Vert x-\varepsilon e_n\Vert \geqslant \vert x(1)\vert >0$. Hence, $\Vert x+\varepsilon e_n\Vert + \Vert x-\varepsilon e_n\Vert>2$, which is a contradiction. 
	\end{proof}
	\smallskip\smallskip
	As in the previous section, for every $i \in \N$, we set $f_i:=T(e_i)$ and $d_i:= T^{-1}(e_i)$.
	
	\begin{lemma}\label{lemma8}
		Let $n\in \N\minus \{1\}$ and $y\notin\{ e_1, -e_1\}$. Then 
		\[ \Vert y+f_n\Vert +\Vert y-f_n\Vert =2 \; \Longleftrightarrow \; y=\pm f_n, \]     
		and
		\[ \Vert y+d_n\Vert +\Vert y-d_n\Vert =2 \; \Longleftrightarrow \; y=\pm d_n.\]
	\end{lemma}
	\begin{proof}
		Lemmas~\ref{lemma3} and~\ref{lemma4} imply that $T(\{ e_1, -e_1\})=\{ e_1, -e_1\}$. Hence, Lemma~\ref{lemma8} follows immediately from Lemmas~\ref{lemma6} and~\ref{lemma7} because $T$ is a surjective isometry. 
	\end{proof}
	\begin{lemma}\label{lemma9}
		Let $n \in \N\minus \{1\}$. For any $i \in \mathrm{supp}(f_n)$ $($resp. $i \in \mathrm{supp}(d_n))$, there exists a $1$-set $F$ of $f_n$ $($resp. $d_n)$ such that $i\in F$. 
	\end{lemma}
	\begin{proof}
		Towards a contradiction, assume that there exists $i\in \mathrm{supp}(f_n)$ such that $i\notin F$ for any $1$-set   $F$  of $f_n$. From Lemma~\ref{lemma5}, we have  $\mathrm{supp}(f_n)\subset\N\minus \{1\}$ and hence $i\neq 1$. Let $y:=e_1+ \varepsilon_{f_n}e_i$ where $\varepsilon_{f_n}>0$ is the $\varepsilon$-gap for $f_n$. Note that the definition of the $\varepsilon$-gap is given in Fact~\ref{normeatteintep1}. It is clear that $y \in \Sp$. For any $F\in \Sc_\alpha$ such that $i\in F$ we have  $1\notin F $ (because $i\neq 1 $ and $\{1\}\in \Sc_\alpha^{MAX}$) and  $F \notin \mathcal{A}_{f_n}^1$, so by the definition of $\varepsilon_{f_n}$, it follows that 
		\[\sum_{k\in F}\vert y(k)\pm f_n(k)\vert = \vert y(i)\pm f_n(i)\vert+\sum_{k\in F\minus \{i\}}\vert f_n(k)\vert \leqslant \varepsilon_{f_n}+ \sum_{k\in F}\vert f_n(k)\vert< 1.\]
		For any $G\in  \Sc_\alpha$ such that $1,i\notin G$, we have  
		\[\sum_{k\in G}\vert y(k)\pm f_n(k)\vert = \sum_{k\in G}\vert f_n(k)\vert\leqslant1. \]
		Therefore, $\Vert y \pm  f_n\Vert =1$ because $\vert y(1)\pm f_n(1)\vert =1$. So, $\Vert y+f_n\Vert +\Vert y-f_n\Vert =2$, where $y\notin \{\pm e_1, \pm f_n\}$ (because $1\notin \mathrm{supp}(f_n)$), which contradicts Lemma~\ref{lemma8}. 
		
		In the same way one shows that for any $i \in \mathrm{supp}(d_n)$, there exists a $1$-set $F$ of $d_n$ such that $i\in F$.      \end{proof}
	\begin{lemma}\label{lemma10}
		For any $n\in \N\minus \{1\}$, $f_n$ and $d_n$ belong to $c_{00}$. 
	\end{lemma}
	\begin{proof}
		By   Fact~\ref{normeatteintep1}, there exists a finite number of $1$-sets for $f_n$, and every $1$-set of $f_n$ is a finite subset of $\N$. Hence, $B:=\{i\in \N; \;  i\in F \;\text{for some $1$-set $F$ of $f_n$}\}$ is a finite set. By Lemma~\ref{lemma9}, we have $\mathrm{supp}(f_n)= B $ and hence $f_n\in c_{00}$. In the same way one shows that $d_n\in c_{00}$.
	\end{proof}
	\begin{lemma}\label{lemma11}
		Let $n,m\in \N\minus \{1\}$. Then, $f_n(m)\neq0 $ if and only if $d_m(n)\neq0$. 
	\end{lemma}
	\begin{proof}
		Let $f_n(m)\neq0 $ and $\varepsilon$ be the sign of $f_n(m)$. Since $m\in \mathrm{supp}(f_n)$, Lemma~\ref{lemma9} implies that there exists $F \in \Sc_\alpha$ such that $m\in F $ and $\sum_{k\in F}\vert f_n(k)\vert =1$. It follows that  \begin{align}\label{a1}\Vert e_m+\varepsilon f_n\Vert =2.\end{align} 
		
		Since $\{m\}$ is the only $1$-set of $e_m$ and $\mathrm{sgn}(-\varepsilon f_n(m))=-1$, Lemma~\ref{lemma1} implies that \begin{align}\label{a2}\Vert e_m-\varepsilon f_n\Vert <2 . \end{align}
		By Lemma~\ref{lemma6}, (\ref{a1}) and (\ref{a2}) imply that $\Vert d_m+ \varepsilon e_n\Vert=2$ and $\Vert d_m- \varepsilon e_n\Vert<2$.
		
		Towards a contradiction, assume that $d_m(n)=0$. Since $\Vert d_m+ \varepsilon e_n\Vert=2$, Lemma~\ref{lemma1} and Fact~\ref{normeatteintep1} imply that there exists $F \in \Sc_\alpha$ such that $n\in F$ and $\sum_{k\in F}\vert d_m(k)\vert=1$. Hence, since $d_m(n)=0$, it follows that 
		\[\Vert d_m- \varepsilon e_n\Vert\geqslant\sum_{k\in F} \vert d_m(k)- \varepsilon e_n(k)\vert =1+\sum_{k\in F}\vert d_m(k)\vert=2,\]
		which is a contradiction because  $\Vert d_m- \varepsilon e_n\Vert<2$. 
		
		In the same way one shows the reverse implication. 
	\end{proof}
	\begin{lemma}\label{lemma12}
		For every $k \in \N\minus\{1\}$, there exists $n>k$ such that $k<\mathrm{supp}(f_n)$ and $k<\mathrm{supp}(d_n)$.   
	\end{lemma}
	\begin{proof}
		The proof is the same as that of Lemma~\ref{l9} (we use Lemmas~\ref{lemma5},~\ref{lemma10} and ~\ref{lemma11} instead of~\ref{l5}, ~\ref{l7} and ~\ref{l8} respectively).   \end{proof} 
	%\begin{lemma}\label{lemma13}    
	%Let $n\in \N\minus \{1\}$ and $y \in \Sp$. Then 
	%  \[ \Vert y+f_n\Vert =\Vert y-f_n\Vert =1 \; \Longleftrightarrow \; y=\pm e_1.  \]     
	%  \end{lemma}
%   \begin{proof}
	%  First, we need the following fact. 
	%    \begin{fact}\label{lemma14}
		%     Let $n\in \N\minus \{1\}$ and $x\in \Sp$. Then 
		%\[ \Vert x+e_n\Vert =\Vert x-e_n\Vert =1 \; \Longleftrightarrow \; x=\pm e_1.  \]      
		%  \end{fact}
	%  \begin{proof}[Proof of Fact~\ref{lemma14}]
		%   The reverse implication is obvious. For the forward implication, assume towards a contradiction that $x\notin \{e_1,-e_1\}$. Then, there exists $m\in \N\minus \{1\}$ such that $x(m)\neq0$ because $x\in \Sp$. If $m=n$, then $\Vert x+\mathrm{sgn}(x(m))e_n\Vert \geqslant1+\vert x(m)\vert >1$, which is a contradiction. If $m\neq n$, then $\{m,n\}\in \Sc_\alpha$ and hence $\Vert x+e_n\Vert \geqslant 1 +\vert x(m)\vert >1$, which is again a contradiction. Therefore $x=\pm e_1$. 
		%\end{proof}
		% Lemmas~\ref{lemma3} and~\ref{lemma4} imply that $T(\{ e_1, -e_1\})=\{ e_1, -e_1\}$. Hence, Lemma~\ref{lemma13} follows immediately from Fact~\ref{lemma14} and Lemma~\ref{lemma6}. 
		% \end{proof}
	\begin{lemma}\label{lemma15}
		Let $n\in  \N\minus \{1\}$. \begin{enumerate}[label={\rm \arabic*)}]
			\item If $2\notin \mathrm{supp}(f_n)$, then there exists $j\in \N\minus \{1\}$ such that $f_n=\pm e_j$. 
			\item If $2\notin \mathrm{supp}(d_n)$, then there exists $j\in \N\minus \{1\}$ such that $d_n=\pm e_j$. \end{enumerate}
	\end{lemma}
	\begin{proof}
		Let  $n\in  \N\minus \{1\}$ be such that $2\notin \mathrm{supp}(f_n)$. We claim that there exists a $1$-set $F$ of $f_n$ such that $\{2\}\cup F \in \Sc_\alpha$. Towards a contradiction, assume that $\{2\} \cup F \notin \Sc_\alpha$ for every $1$-set $F$ of $f_n$. We define $y:= e_1+ \varepsilon_{f_n} e_2$, where   $\varepsilon_{f_n}$ is the $\varepsilon$-gap for $f_n$. 
		
		For any $F \in \mathcal{A}_{f_n}^1$, we have $1\notin F$ (by Lemma~\ref{lemma5} and since $\{1\}\in \Sc_\alpha^{MAX}$) and  $2\notin F$ (since otherwise we get $\{2\}\cup F'\in \Sc_\alpha$ for some $1$-set $F'$ of $f_n$) and hence 
		\[ \sum_{k\in F}\vert y (k)\pm f_n(k)\vert =\sum_{k\in F}\vert f_n(k)\vert =1.\]
		For any $G\in \Sc_\alpha\minus \mathcal{A}_{f_n}^1$, we have 
		\[\begin{aligned}&\sum_{k\in G}\vert y (k)\pm f_n(k)\vert =\vert y(1)\vert =1 &&\quad\text{if  $G=\{1\}$}, \\
			&\sum_{k\in G}\vert y (k)\pm f_n(k)\vert = \sum_{k\in G}\vert  f_n(k)\vert <1 &&\quad\text{if $1,2\notin G$},\\ 
			&\sum_{k\in G}\vert y (k)\pm f_n(k)\vert= \varepsilon_{f_n}+\sum_{k\in G}\vert  f_n(k)\vert<1 &&\quad\text{if $2\in G$}.\end{aligned}\]
		Therefore, $\Vert y \pm  f_n\Vert =1$. So, $\Vert y+f_n\Vert +\Vert y-f_n\Vert =2$ and $y\notin \{\pm e_1, \pm f_n\}$ (because $1\notin \mathrm{supp}(f_n)$), which contradicts Lemma~\ref{lemma8} and proves our claim. 
		
		%Therefore, $\Vert y\pm f_n\Vert=1$ and $y\notin\{e_1,-e_1\}$, which contradicts Lemma~\ref{lemma14} and proves our claim. 
		
		Choose now  a $1$-set $F_0$ of $f_n$ such that $\{2\}\cup F_0 \in \Sc_\alpha$. We will show that $F_0=\mathrm{supp}(f_n)$. Since $F_0$ is a $1$-set of $f_n$, we know that $F_0\subseteq\mathrm{supp}(f_n)$. Towards a contradiction, assume that there exists $j\in \mathrm{supp}(f_n)\minus F_0$. Since $\{2\}\cup F_0 \in \Sc_\alpha$ and $j>2$ (because $1,2\notin \mathrm{supp}(f_n)$), it follows that $\{j\}\cup F_0\in \Sc_\alpha$ because $\Sc_\alpha$ is spreading. Hence, \[\Vert f_n\Vert \geqslant \vert f_n(j)\vert +\sum_{k\in F_0}\vert f_n(k)\vert= 1+\vert f_n(j)\vert>1, \] which is a contradiction. Therefore, $\mathrm{supp}(f_n)=F_0$. 
		
		It remains to show that $F_0$ is a singleton. Towards a contradiction, assume that  $F_0$ contains at least two elements, and let us fix $i_0\in F_0$. Let $z:= -f_n(i_0)e_{i_0}+ \sum_{i\in F_0\minus \{i_0\}}f_n(i)e_i$. Since $\sum_{k\in F_0} \vert f_n(k)\vert =1$, it follows that $z\in \Sp$. Moreover, 
		\[\Vert z+ f_n\Vert +\Vert z- f_n\Vert = 2 \sum_{i\in F_0\minus \{i_0\}}\vert f_n(i)\vert + 2  \vert f_n(i_0)\vert = 2 \sum_{i\in F_0} \vert f_n(i)\vert=2,\]
		which contradicts Lemma~\ref{lemma8} because $z\notin \{\pm e_1,\pm f_n\}$. Therefore, there exists $j\in \N\minus \{1\}$ such that $\mathrm{supp}(f_n)=F_0=\{j\}$ and hence $f_n =\pm e_j$.  
		
		In the same way, we show 2).
	\end{proof}
	\begin{lemma}\label{lemma16}
		For any $k\geqslant 2$, there exist $i,j>k$ such that $f_i=\pm e_j$ and $d_j=\pm e_i$.
	\end{lemma}
	\begin{proof}
		Let $k\in \N\minus \{1\}$. By Lemma  ~\ref{lemma12}, there exists $i>k$ such that $k<\mathrm{supp}(f_i)$. Since $k\geqslant 2$, it follows that $2\notin \mathrm{supp}(f_i)$. Hence, by Lemma~\ref{lemma15}, there exists $j$ such that $f_i=\pm e_j$, and so $j>k$ because $k<\mathrm{supp}(f_i)$. This implies, by Lemma~\ref{lemma6}, that $d_j=\pm e_i$.   
	\end{proof}
	\begin{lemma}\label{nonmaxSalpha}
		If $n\in \N\minus \{1\}$, then $f_n$ has  a non-maximal $1$-set \textit{i.e.} there exists a $1$-set $F$ of $f_n$ such that $F\in  \Sc_\alpha\minus \Sc_\alpha^{MAX}$.
	\end{lemma}
	\begin{proof}
		By Lemma~\ref{lemma16}, we can choose $i,j>\{n\}\cup \mathrm{supp}(f_n)$ such that $f_i=\theta e_j$ where $\theta$ is a sign. Since $\Vert e_n+e_i\Vert =2$, it follows that $\Vert f_n+\theta e_j\Vert=2$ and so by Lemma~\ref{lemma1} there exists $G\in \mathcal{A} _{f_n}^1$ such that $j\in G$ and $\sum_{k\in G}\vert f_n(k)+\theta e_j(k)\vert=2$. Then $F:=G\cap \mathrm{supp}(f_n)$ is a $1$-set of $f_n$ and $j\notin F$ because $j>\mathrm{supp}(f_n)$. Hence $F\in \Sc_\alpha\minus \Sc_\alpha^{MAX}$ since  $F\cup \{j\}\in \Sc_\alpha$ (because  $F\cup \{j\}\subseteq G \in \Sc_\alpha$ and $\Sc_\alpha$ is hereditary). 
	\end{proof}
	
	\begin{lemma}\label{nonmaxS1}
		Assume that $\alpha=1$. 
		Let $x\in \Sp$ and $F$ be a non-maximal $1$-set of $x$. Then, 
		$F$ is the only non-maximal  $1$-set of $x$ and it has the form 
		\[F=[\min (F), \infty)\cap \mathrm{supp}(x). \]
	\end{lemma}
	We will need the following simple fact concerning $\Sc_1$. 
	\begin{fact}\label{factS1}
		Let $F\in \Sc_1\minus \Sc_1^{MAX}$. Then $F\cup \{n\}\in \Sc_1$ for any $n>\min(F)$. 
	\end{fact}
	\begin{proof}[Proof of Lemma~\ref{nonmaxS1}]
		Since $x$ has a non-maximal $1$-set, it follows that  $x\in c_{00}$ since otherwise all the $1$-sets of $x$ would  be maximal. Assume towards a contradiction that there exists $n\in \mathrm{supp}(x)\minus F$ such that $n>\min(F)$. By Fact~\ref{factS1}, $F\cup \{n\}\in \Sc_1$ and so 
		\[\Vert x\Vert \geqslant \sum_{k\in F}\vert x(k)\vert + \vert x(n)\vert =1+\vert x(n)\vert >1,\] which is a contradiction. Hence, $F=[\min (F), \infty)\cap \mathrm{supp}(x)$. 
		
		Assume now that there exists  a non-maximal $1$-set $F'$ of $x$, distinct from $F$. Then $F'=[\min (F'), \infty)\cap \mathrm{supp}(x)$. Since $F'\neq F$, it follows  that $\min(F')>\min(F)$ or $\min(F')<\min(F)$ which contradicts the fact that $F$ and $F'$ are $1$-sets of $x$. 
	\end{proof}
	\begin{remark}
		When $\alpha=1$,  for any  $k\in \N\minus \{1\}$, $f_k$ has exactly one non-maximal $1$-set  by Lemmas~\ref{nonmaxSalpha} and~\ref{nonmaxS1}. 
	\end{remark}
	\begin{lemma}\label{intersectionS1}
		Assume that $\alpha=1$. 
		Let $k \in \N\minus \{1\}$ and let $F$ be the non-maximal $1$-set of $f_k$. If $G$ is a $1$-set of $f_k$ then $F\cap G\neq \emptyset$. 
	\end{lemma}
	First we need the following fact that is true for any  $1\leqslant\alpha<\omega_1$.
	\begin{fact}\label{lemma20}
		Let $i<j \in \N\minus \{1\}$ and $x\in \Sp$. If $\Vert x+e_i\Vert =\Vert x-e_i\Vert =2$, then $\Vert x+e_j\Vert =2$ or $\Vert x-e_j\Vert =2$.
	\end{fact}
	\begin{proof}[Proof of Fact~\ref{lemma20}]
		Towards a contradiction, assume that $i\in \mathrm{supp}(x)$. Since $\{i\}$ is the only $1$-set of $e_i$ and since  $\Vert x+e_i\Vert =\Vert x-e_i\Vert =2$, it follows by Lemma~\ref{lemma1} that $x(i)>0$ and $x(i)<0$ which is a contradiction. Hence, $x(i)=0$. 
		
		Since, $\Vert x+e_i\Vert =2$, there exists $F \in \mathcal{A}_{x}^1$ such that $i \in F$ and $\sum_{k\in F }\vert x(k)+e_i(k)\vert =2$. This implies that 
		\[\sum_{k\in F\minus \{i\}}\vert x(k)\vert =1,\] 
		because $x(i)=0$. 
		
		Since $\Sc_\alpha$ is spreading and $j>i$, we have $(F\minus \{i\})\cup \{j\}\in \Sc_\alpha$, and so it is easy to verify that $\Vert x+e_j\Vert =2 $ or $\Vert x-e_j\Vert =2 $.  
	\end{proof}
	\begin{proof}[Proof of Lemma ~\ref{intersectionS1}]
		Assume towards a contradiction that $F \cap G =\emptyset$. Since $F$ is the non-maximal $1$-set of $f_k$, it follows by Lemma~\ref{nonmaxS1} that $r:=\max\mathrm{supp}(f_k)\in F$. Since $F \in \Sc_1\minus \Sc_1^{MAX}$, it follows  that there exist $r_1,\dots, r_s \in \N$ where $s\geqslant1$, $r<r_1<\dots<r_s$ and $F':=F\cup \{r_1,\dots, r_s\}\in \Sc_1^{MAX}$. Now, let us define $z\in X_{\Sc_1}$ as follows: 
		\[z:=-\sum_{i\in G}f_k(i)e_i+\sum_{i\in F\minus \{r\}}f_k(i)e_i+ \frac{f_k(r)}{s+1}e_r+\sum_{i=1}^s\frac{f_k(r)}{s+1}e_{r_i}.\]
		We will show first that $z\in \Sp$. Since $\sum_{i\in G}\vert z(i)\vert =1$ and $G\in \Sc_1$, it is enough to show that $\sum_{i\in H'}\vert z(i)\vert \leqslant1$ for every $H'\in \Sc_1$. So let us fix $H'\in \Sc_1$. Also, let us set $H:=F \cup G$. We have
		\begin{align}\label{sjs} \sum_{i\in H'\minus H}\vert z(i)\vert \leqslant \sum_{i=1}^s\vert z(r_i)\vert= \frac{s}{s+1}\vert f_k(r)\vert =\vert f_k(r)\vert -\frac{\vert f_k(r)\vert}{s+1}=\vert f_k(r)\vert-\vert z(r)\vert.\end{align}
		If $r\in H'$, then \[\begin{aligned}
			\sum_{i\in H'}\vert z(i)\vert &=\sum_{i\in H'\cap(H\minus \{r\})}\vert z(i)\vert +\vert z(r)\vert +\sum_{i\in H'\minus H}\vert z(i)\vert \\&\leqslant \sum_{i\in H'\cap(H\minus \{r\})}\vert z(i)\vert+ \vert f_k(r)\vert \qquad\text{by (\ref{sjs})}\\&=\sum_{i\in H'\cap H}\vert f_k(i)\vert \leqslant 1,
		\end{aligned}\]
		because $H'\cap H\in \Sc_1$ (since $H'\cap H\subseteq H'\in\Sc_1 $ and $\Sc_1$ is hereditary) and $\Vert f_k\Vert=1$.

		Now, we will assume that $r\notin H'$ and we will distinguish two cases: $H'\cap H\in \Sc_1^{MAX}$ or $H'\cap H\in \Sc_1\minus\Sc_1^{MAX}$.
		
		If $H'\cap H\in \Sc_1^{MAX}$, then $H'=H'\cap H$ (because $H'\cap H\subseteq H'$ and $H'\in \Sc_1$). Hence, since we are assuming that $r\notin H'$, it follows that 
		\[\sum_{i\in H'}\vert z(i)\vert = \sum_{i\in H'\cap H}\vert z(i)\vert= \sum_{i\in H'\cap H}\vert f_k(i)\vert\leqslant 1. \]
		
		If $H'\cap H\in \Sc_1\minus \Sc_1^{MAX}$, then $(H'\cap H)\cup \{r\}\in \Sc_1$ (by Fact~\ref{factS1} since $r>H'\cap H$). Hence, since $\Vert f_k\Vert=1$, we have 
		\begin{align}\label{sj} \sum_{i\in H'\cap H}\vert f_k(i)\vert +\vert f_k(r)\vert \leqslant 1. \end{align}
		So, since $r\notin H'$, it follows that 
		\[\begin{aligned} \sum_{i\in H'}\vert z(i)\vert &=\sum_{i\in H'\cap H} \vert z(i)\vert +\sum_{i\in H'\minus H}\vert z(i)\vert \\&=\sum_{i\in H'\cap H} \vert f_k(i)\vert +\sum_{i\in H'\minus H}\vert z(i)\vert\\&\leqslant 1-\vert f_k(r)\vert + \sum_{i\in H'\minus H}\vert z(i)\vert \qquad \text{by (\ref{sj})}\\&\leqslant 1-\vert z(r)\vert <1 \qquad\text{by (\ref{sjs})}.
		\end{aligned}\]
		Therefore, $\Vert z\Vert =1$. 
		
		Moreover, $\Vert z\pm f_k\Vert =2$ because 
		\[ \Vert z-f_k\Vert \geqslant \sum_{i\in G}\vert z(i)-f_k(i)\vert =2\sum_{i\in G}\vert f_k(i)\vert =2\]
		and 
		\[\Vert z+f_k\Vert \geqslant \sum_{i\in F'}\vert z(i)+f_k(i)\vert =2\sum_{i\in F}\vert f_k(i)\vert =2. \]
		
		Choose now $n,m > \max\{r_s,k\}$ such that $f_n=\pm e_m$ by Lemma~\ref{lemma16}. 
		
		Since $\Vert z\pm f_k\Vert=2$, it follows that $\Vert T^{-1}(z)\pm e_k\Vert =2$ by Lemma~\ref{lemma6}. Hence, by Fact~\ref{lemma20} and since $k<n$, we have $\Vert T^{-1}(z)+ e_m\Vert =2$  or $\Vert T^{-1}(z)- e_n\Vert =2$. So,  $\Vert z + f_n\Vert =2$ or $\Vert z-f_n\Vert =2$ by Lemma~\ref{lemma6}. Since $f_n=\pm e_m$, it follows that  \begin{align}\label{hypothese}
			\Vert z-e_m\Vert =2 \qquad\text{or} \qquad \Vert z+e_m\Vert=2.
		\end{align}
		By Lemma~\ref{lemma1} and Fact~\ref{normeatteintep1}, and since $m>r_{s}$, (\ref{hypothese}) implies that there exists a $1$-set $H''$ of $z$ such that $H''\cup \{m\}\in \Sc_1$. Hence, $H''$ is a non maximal $1$-set  of $z$. By Lemma~\ref{nonmaxS1}, 
		\[H''=[\min (H''), \infty)\cap \mathrm{supp}(z).\]
		Note also that $F'$ is a $1$-set of $z$ such that 
		\[F'= [\min (F'), \infty)\cap \mathrm{supp}(z)\]
		since  $F=[\min (F), \infty)\cap \mathrm{supp}(f_k)$ because $F$ is the non-maximal $1$-set of $f_k$.
		
		It follows that if $\min (H'')>\min (F')$ or $\min (H'')<\min (F')$, then we have a contradiction with the fact that $F'$ and $H''$ are $1$-sets of $z$. Therefore, $\min (H'')=\min (F')$ and $H''=F'$ which is a contradiction because $F'\in \Sc_1^{MAX}$ and $H''\in \Sc_1\minus \Sc_1^{MAX}$. 
	\end{proof}

					Recall that $A$ is the set of all  $x\in \Sp$ such that $1\in \mathrm{supp}(x)$.
					
					\begin{lemma}\label{simple}
						Assume that $\alpha=1$. Let $x,y \in \Sp\minus A$ be such that  $\Vert x\pm e_2\Vert=\Vert y\pm e_2\Vert=2$. Then $x= y$ or $x=-y$ or $\Vert x+y\Vert=\Vert x-y\Vert=2$.   
					\end{lemma}
					\begin{proof}
						Assume towards a contradiction that $x(2)\neq 0$. Since $\Vert x\pm e_2\Vert =2$, it follows by Lemma~\ref{lemma1} that $x(2)>0$ and $x(2)<0$ which is a contradiction. Hence, $x(2)=0$ and in the same way we have $y(2)=0$. 
						
						Now, we will show that $x=\pm e_m$ for some $m\in \N\minus \{1,2\}$. By Lemma~\ref{lemma1} again and since $\Vert x+ e_2\Vert =2$, there exists $F\in \mathcal{A}_{x}^1$ such that $2\in F$. Note that $F\neq \{2\}$ because $x(2)=0$. So, $F=\{2,m\}$ for some $m>2$ because $F\in \Sc_1$. It follows that $\vert x(m)\vert=1$. Consequently, $x(i)=0$ for all $i\in \N\minus \{1,m\}$ because $\Vert x\Vert =1$ and $x(1)=0$ because $x\in \Sp\minus A$. Therefore, $x=\pm e_m$. 
						
						In the same way, we show that  $y=\pm e_n$ for some $n\in \N\minus \{1,2\}$.
						
						If $n=m$ then $x=y$ or $x=-y$. And if $n\neq m$, then $\Vert x+y\Vert=\Vert x-y\Vert=2$. 
					\end{proof}
					\begin{lemma}\label{lemma23}
						Assume that $\alpha\geqslant2$.  Let $i,j \in \N\minus \{1\}$ be such that $i\neq j$ and let $x \in \Sp\minus A $. 
						Then  
						\[ (\Vert x+e_i\Vert =\Vert x+e_j\Vert =2 \,\;\text{ and }\,\; \Vert x- e_i\Vert =\Vert x-e_j\Vert =1 )\Longleftrightarrow \; x=\frac{1}{2}e_i+\frac{1}{2}e_j.  \]  
					\end{lemma}
					\begin{proof}
						
						The reverse implication is obvious since $\{i,j\}\in \Sc_\alpha$. 
						
						For the forward implication, since $\Vert x+e_i\Vert =2$, it follows by Lemma~\ref{lemma1} and Fact~\ref{normeatteintep1} that there exists $F\in \mathcal{A}_{x}^1$ such that $i\in F$ and $\sum_{k\in F}\vert x(k)+e_i(k)\vert=2$. 
						
						If $x(i)=0$, then $F\minus \{i\}\in \mathcal{A}_{x}^1$ and so  \[\begin{aligned}
							\Vert x-e_i\Vert &\geqslant \sum_{k\in F}\vert x(k)-e_i(k)\vert \\&= \sum_{k\in F\minus \{i\}}\vert x(k)\vert +1=2,
						\end{aligned}\]
						which is a contradiction, because $\Vert x-e_i\Vert =1$. 
						Hence $x(i)\neq 0$ and so $x(i)>0$ by Lemma~\ref{lemma1}.
						
						Since $\Vert x-e_i\Vert =1$ and $\sum_{k\in F\minus \{i\}}\vert x(k)\vert =1- x(i) $ (because $F\in \mathcal{A}_x^1$), it follows that 
						\[ \begin{aligned}
							\sum_{k\in F}\vert x(k)-e_i(k)\vert \leqslant 1 &\implies 1-x(i)+\sum_{k\in F\minus \{i\}}\vert x(k)\vert \leqslant 1\\&\implies 2-2 x(i)\leqslant 1\\&\implies  x(i) \geqslant \frac{1}{2}\cdot
						\end{aligned}\]
						
						In the same way, we show that  $ x(j)\geqslant \frac{1}{2}\cdot$ 
						
						If $ x(i)>\frac{1}{2}$ or $ x(j) >\frac{1}{2}$, then since $\{i,j\}\in \Sc_\alpha$ we get
						\[\Vert x\Vert \geqslant x(i) + x(j) >1,\]
						which is a contradiction because $x\in \Sp$. 
						Hence $x(i)=x(j)=\frac{1}{2}\cdot$
						
						Assume now towards a contradiction that there exists $m\in\N\minus\{1,i,j\}$ such that $x(m)\neq0$. Since $\alpha\geqslant 2$, it follows that $\{m,i,j\}\in \Sc_\alpha$ and so 
						\[\Vert x\Vert \geqslant \vert x(m)\vert +\frac{1}{2}+\frac{1}{2}>1,\]
						which is a contradiction. Hence, $x(m)=0$ for all $m\in\N\minus\{1,i,j\}$ and $x(1)=0 $ because $x\in \Sp\minus A$. Therefore, $x=\frac{1}{2}e_i+\frac{1}{2}e_j$. 
					\end{proof}  
					
					\begin{lemma}\label{treslong}
						We have $d_2=\pm e_j$ for some $j\in \N\minus\{1\}$.
					\end{lemma}
					\begin{proof}
						If $2\notin \mathrm{supp}(d_2)$, then $d_2=\pm e_j$ for some $j\in \N\minus\{1\}$, by Lemma~\ref{lemma15}. 
						
						\smallskip
						In the rest of the proof, we will suppose that $2\in \mathrm{supp}(d_2)$ and we will assume towards a contradiction that there exists $k \in \mathrm{supp}(d_2) $ such that $k\neq 2$. Note that $k\in \N\minus \{1\}$ by Lemma~\ref{lemma5}. 
						
						We will show first that 
						\begin{align}\label{suppfk}
							\mathrm{supp}(f_k)\subseteq\mathrm{supp}(f_2).
						\end{align}
						Towards a contradiction, assume that there exists $k'\in \mathrm{supp}(f_k)$ such that $k'\notin \mathrm{supp}(f_2)$. We have $k'\in \N\minus \{1\}$ by Lemma~\ref{lemma5} and $2\notin \mathrm{supp}(d_{k'})$ by Lemma~\ref{lemma11}. Hence, $d_{k'}=\pm e_j$ for some $j\in \N\minus \{1\}$ by Lemma~\ref{lemma15}.   Since $k'\in \mathrm{supp}(f_k)$, we have $\Vert e_{k'}+f_k\Vert <2$ or  $\Vert e_{k'}-f_k\Vert <2$ by Lemma~\ref{lemma1}. Hence, $\Vert d_{k'}+e_k\Vert<2$ or  $\Vert d_{k'}-e_k\Vert<2$ by Lemma \ref{lemma6}, and so $\Vert e_j+e_k\Vert <2$ or $\Vert e_j-e_k\Vert <2$. This implies that $j=k$ since $j,k\in \N\minus \{1\}$ and hence $\{j,k\}\in \Sc_\alpha$. Therefore, $d_{k'}=\pm e_k$ which means $f_k=\pm e_{k'}$ by Lemma~\ref{lemma6}. However, since $k\in \mathrm{supp}(d_2)$, it follows by Lemma~\ref{lemma11} that $2\in \mathrm{supp}(f_k)$. So, $k'=2$ which is a contradiction because $k'\notin \mathrm{supp}(f_2)$ and $2\in \mathrm{supp}(f_2)$ (by Lemma~\ref{lemma11} since we are assuming that $2\in \mathrm{supp}(d_2)$). This proves (\ref{suppfk}).
						
						\smallskip
						Now, we will distinguish two cases. 
						
						\smallskip
						\textbf{Case 1:} $\alpha=1$. 
						
						\smallskip
						Since $\Vert e_k\pm e_2\Vert=2$, it follows that $\Vert f_k\pm f_2\Vert=2$. Hence, by Lemma~\ref{lemma1}, there exists $F,G\in \mathcal{A}_{f_k}^1\cap \mathcal{A}_{f_2}^1$ such that 
						\begin{align}\label{mm1}
							\sum_{i\in F}\vert f_k(i)+f_2(i)\vert =2
						\end{align}
						and 
						\begin{align}\label{mm2}
							\sum_{i\in G}\vert f_k(i)- f_2(i)\vert =2.
						\end{align}
						Hence, $F':=F\cap \mathrm{supp}(f_k)$ and $G':=G \cap \mathrm{supp}(f_k)$ are two $1$-sets of $f_k$. 
						
						If there exists $i_0\in F'\cap G'$, then, by Lemma~\ref{lemma1} and since $\mathrm{supp}{f_k}\subseteq\mathrm{supp}(f_2)$ (by (\ref{suppfk})),  (\ref{mm1}) implies that 
						$\mathrm{sgn}(f_k(i_0))=\mathrm{sgn}(f_2(i_0))$ and (\ref{mm2}) implies that 
						$\mathrm{sgn}(f_k(i_0))=-\mathrm{sgn}(f_2(i_0))$
						which is a contradiction. Hence, $F'\cap G'=\emptyset$ and so by Lemma~\ref{intersectionS1}, $F', G'\in \Sc_1^{MAX}$. Therefore, $F'=F$ and $G'=G$.   In summary, 
						\begin{align}\label{resume}
							\text{$F,G\in \Sc_1^{MAX}$ are two $1$-sets of $f_k$ and $f_2$ such that $F\cap G=\emptyset$.}   
						\end{align}
						Now, let us define $z\in X_{\Sc_1}$ as follows: 
						\[z:=\sum_{i\in F}\frac{f_k(i)+f_2(i)}{2}e_i+\sum_{i\in G}\frac{f_k(i)-f_2(i)}{2}e_i.\]
						Since, by Lemma~\ref{lemma1},  $\mathrm{sgn}(f_k(i))=\mathrm{sgn}(f_2(i))$ for all $i\in F$, and $\mathrm{sgn}(f_k(i))=-\mathrm{sgn}(f_2(i))$ for all $i\in G$, it follows that for any $i\in F\cup G$, we have 
						\begin{align}\label{ll}\vert z(i)\vert =\frac{\vert f_k(i)\vert+\vert f_2(i)\vert}{2}\cdot\end{align}
						Now, we will show that $z\in \Sp$. By (\ref{mm1}) and (\ref{mm2}), we have  
						\[\sum_{i\in F}\vert z(i)\vert =\sum_{i\in G}\vert z(i)\vert=1.\]
						Moreover, for any $H \in \Sc_1$, we have by (\ref{ll}) and since $F\cap G=\emptyset$ that
						\[\begin{aligned}\sum_{i\in H}\vert z(i)\vert &=\sum_{i \in H\cap F}\vert z(i)\vert +
							\sum_{i \in H\cap G}\vert z(i)\vert \\&=\frac{1}{2}\sum_{i \in H\cap F}(\vert f_k(i)\vert +\vert f_2(i)\vert )+\frac{1}{2}
							\sum_{i \in H\cap G}(\vert f_k(i)\vert +\vert f_2(i)\vert )\\&=\frac{1}{2}\sum_{i\in H\cap(F\cup G)}\vert f_k(i)\vert +\frac{1}{2}\sum_{i\in H\cap(F\cup G)}\vert f_2(i)\vert \leqslant\frac{1}{2}+\frac{1}{2}=1,\end{aligned}\]
						because $H\cap (F\cup G)\in \Sc_1$ (since $H\in \Sc_1$ and $\Sc_1$ is hereditary). 
						
						Therefore $z\in \Sp$. 
						
						Since $\mathrm{sgn}(f_k(i))=\mathrm{sgn}(f_2(i))$ for all $i\in F$, it follows that 
						\[\begin{aligned}
							\Vert z+f_2\Vert &\geqslant \sum_{i\in F}\vert z(i)+f_2(i)\vert \\&=\frac{1}{2}\sum_{i\in F}\vert f_k(i)+3f_2(i)\vert \\&=\frac{1}{2}\sum_{i\in F}\vert f_k(i)\vert +\frac{3}{2}\sum_{i\in F}\vert f_2(i)\vert =2,
						\end{aligned}\]
						and since $\mathrm{sgn}(f_k(i))=-\mathrm{sgn}(f_2(i))$ for all $i\in G$, we get in the same way that 
						\[\Vert z-f_2\Vert \geqslant\sum_{i\in G}\vert z(i)-f_2(i)\vert=2. \]
						Hence, $\Vert z\pm f_2\Vert =2$. So, $x:=T^{-1}(z)$ satisfies  $\Vert x \pm e_2\Vert =2=\Vert e_k\pm e_2\Vert$, and $x\in \Sp\minus A$ by Lemma~\ref{lemma5}. By Lemma~\ref{simple}, it follows that 
						\[x=e_k \qquad\text{or}\qquad x=-e_k \qquad\text{or}\qquad \Vert x+e_k\Vert =\Vert x-e_k\Vert =2. \]
						Hence 
						\[z=f_k \qquad\text{or}\qquad z=-f_k \qquad\text{or}\qquad \Vert z+f_k\Vert =\Vert z-f_k\Vert =2. \]
						However, $z\neq -f_k$, and $\Vert z-f_k\Vert <2$ since for any $i\in F\cup G$, $\mathrm{sgn}(f_k(i))=\mathrm{sgn}(z(i))$. Therefore, $z=f_k$ and this implies that 
						\[\begin{aligned} &f_2(i)=f_k(i) &&\qquad\text{for all}\qquad i\in F,\\
							&f_2(i)=-f_k(i) &&\qquad\text{for all}\qquad i\in G,\end{aligned}\]
						and \[f_k=\sum_{i\in F}f_k(i)e_i+\sum_{i\in G}f_k(i)e_i\]
						where $F,G$ are two maximal $1$-sets of $f_k$.

						Recall now that $k\in \mathrm{supp}(d_2)$. Hence, by Lemma~\ref{lemma11}, $2\in \mathrm{supp}(f_k)= F\cup G$. Without loss of generality, we will assume that $2\in F$. Since $F\in \Sc_1^{MAX}$, there exists $m\in \N\minus\{1\}$ such that 
						\[F=\{2,m\}. \]
						
						Let $H'$ be the non-maximal $1$-set of $f_k$ (which is well defined by Lemmas~\ref{nonmaxSalpha} and~\ref{nonmaxS1}). Then 
						\begin{align}\label{ccc}H'=[\min(H'),\infty)\cap\mathrm{supp}(f_k). \end{align}
						By Lemma~\ref{intersectionS1}, $H'\cap F\neq  \emptyset$. Since $F=\{2,m\}$, it follows that $2\in H'$ or $m\in H'$. If $2\in H'$, then by (\ref{ccc}), $H'=F\cup G$ and hence since $F\cap G=\emptyset$ we have\[\sum_{i\in H'}\vert f_k(i)\vert=\sum_{i\in F}\vert f_k(i)\vert +\sum_{i\in G}\vert f_k(i)\vert =2,\]
						which is a contradiction. Hence, $2\notin H'$, and so $m\in H'$ and we have by (\ref{ccc}) that 
						\[[m,\infty)\cap\mathrm{supp}(f_k)\subseteq H',\]
						which implies that 
						\begin{align}\label{ccc1}
							[m,\infty)\cap G\subseteq H',  
						\end{align}
						since $\mathrm{supp}(f_k)=F\cup G$. 
						
						If $G\subset H'$, then $G\cup \{m\}\subseteq H'\in \Sc_1$, and hence $G\cup\{m\}\in \Sc_1$ because $\Sc_1$ is hereditary, which is a contradiction since $G\in \Sc_1^{MAX}$ and $m\notin G$. So, there exists $n\in G$ such that $n\notin H'$. Therefore, by (\ref{ccc1}), $n<m$. 
						
						We will show that $\vert f_k(n)\vert =\vert f_k(m)\vert $. Towards a contradiction, assume first that $\vert f_k(n)\vert <\vert f_k(m)\vert $. Since $n<m$ and $\Sc_1$ is spreading, it follows that $(G\minus\{n\})\cup\{m\}\in \Sc_1$. Moreover, since $G$ is a $1$-set of $f_k$, it follows that 
						\[\sum_{i \in G\minus \{n\}}\vert f_k(i)\vert =1-\vert f_k(n)\vert.\]
						Hence, 
						\[\begin{aligned}\Vert f_k\Vert &\geqslant \sum_{i\in (G\minus\{n\})\cup\{m\}}\vert f_k(i)\vert\\&=\sum_{i\in G\minus \{n\} }\vert f_k(i)\vert +\vert f_k(m)\vert \\&=1-\vert f_k(n)\vert +\vert f_k(m)\vert>1,
						\end{aligned}\]
						since we are assuming that $\vert f_k(n)\vert <\vert f_k(m)\vert $, which is a contradiction. Hence, $\vert f_k(m)\vert \leqslant \vert f_k(n)\vert $.
						
						Assume now that $\vert f_k(m)\vert< \vert f_k(n)\vert $. Since $F=\{2,m\}$ is a $1$-set of $f_k$, it follows that 
						\[\vert f_k(2)\vert =1-\vert f_k(m)\vert,\]
						and so since $\{2,n\}\in \Sc_1$, we have  \[\Vert f_k\Vert \geqslant \vert f_k(2)\vert+\vert  f_k(n)\vert=1-\vert  f_k(m)\vert +\vert  f_k(n)\vert>1, \]
						which is a contradiction. Hence $\vert f_k(n)\vert =\vert f_k(m)\vert$ and so 
						\[\vert f_k(2)\vert+ \vert f_k(n)\vert =\vert f_k(2)\vert+\vert f_k(m)\vert =1,\]
						which implies that $\{2,n\}$ is a $1$-set of $f_k$. However, $\{2,n\}\cap H'=\emptyset$, which contradicts Lemma~\ref{intersectionS1} because $H'$ is the non-maximal $1$-set of $f_k$.

						\smallskip
						\textbf{Case 2:} $\alpha\geqslant2$.
						
						\smallskip
						By Lemma~\ref{lemma12}, we may choose $n,m> \{k\}\cup \mathrm{supp}(f_k)\cup \mathrm{supp}(f_2)$ such that $f_n=\theta e_m$ where $\theta=\pm 1$. Let us define $z\in X_{\Sc_\alpha}$ as follows: 
						\[z:=\frac{1}{2}f_k+\frac{1}{2}f_n. \]
						
						Let us show that 
						\begin{align}\label{mmm} 
							z\in \Sp, \quad \Vert z+ f_k\Vert =\Vert z+f_n\Vert =2\quad \text{and}\quad \Vert z- f_k\Vert =\Vert z-f_n\Vert =1. \end{align}
						By Lemma~\ref{nonmaxSalpha}, there exists a non maximal $1$-set $F_0$ of $f_k$. Since $F_0\in \Sc_\alpha\minus \Sc_\alpha^{MAX}$ and $m>F_0$ (because $m>\mathrm{supp}(f_k)$), it follows by Fact~\ref{max} that  $F_0\cup \{m\}\in \Sc_\alpha$.
						
						We have 
						\[\Vert z\Vert \leqslant\frac{1}{2}\Vert f_k\Vert + \frac{1}{2}\Vert f_n\Vert=1.\]
						Moreover, since $F_0$ is a $1$-set of $f_k$, $F_0\cup \{m\}\in \Sc_\alpha$ and $f_n=\theta e_m$, it follows that  
						\[\Vert z\Vert\geqslant \frac{1}{2}\sum_{i\in F_0}\vert f_k(i)\vert+\frac{1}{2}\vert f_n(m)\vert =1.\]
						Hence, $z\in \Sp$. 
						
						In the same way, we have  
						\[\Vert z+f_k\Vert =\biggl\Vert \frac{3}{2}f_k+\frac{1}{2}f_n\biggr\Vert \geqslant \frac{3}{2}\sum_{i\in F_0}\vert f_k(i)\vert +\frac{1}{2}\vert f_n(m)\vert =2,\]
						and  
						\[\Vert z+f_n\Vert =\biggl\Vert \frac{1}{2}f_k+\frac{3}{2}f_n\biggr\Vert \geqslant \frac{1}{2} \sum_{i\in F_0}\vert f_k(i)\vert +\frac{3}{2}\vert f_n(m)\vert =2. \]
						So, 
						$\Vert z+ f_k\Vert =\Vert z+f_n\Vert =2$. 
						
						Finally, since $\mathrm{supp}(f_n)=\{m\}>\mathrm{supp}(f_k)$, it follows that 
						\[\Vert z-f_k\Vert =\biggl\Vert -\frac{1}{2}f_k+\frac{1}{2}f_n\biggr\Vert= \biggl\Vert \frac{1}{2}f_k+\frac{1}{2}f_n\biggr\Vert=\Vert z\Vert=1,\]
						and \[\Vert z-f_n\Vert =\biggl\Vert \frac{1}{2}f_k-\frac{1}{2}f_n\biggr\Vert=\biggl\Vert \frac{1}{2}f_k+\frac{1}{2}f_n\biggr\Vert=\Vert z\Vert=1.\]
						
						So, we have shown (\ref{mmm}). By Lemma~\ref{lemma6}, it follows that 
						\[\Vert T^{-1}(z)+ e_k\Vert =\Vert T^{-1}(z)+e_n\Vert =2\qquad \text{and}\qquad \Vert T^{-1}(z)- e_k\Vert =\Vert T^{-1}(z)-e_n\Vert =1.\]
						Moreover, $T^{-1}(z)\in \Sp\minus A$ by Lemma~\ref{lemma5}.
						Therefore, by Lemma~\ref{lemma23}, we have 
						\[T^{-1}(z)=\frac{1}{2}e_k+\frac{1}{2}e_n.\]
						Since $\alpha\geqslant2$, it follows that $\{2,k,n\}\in \Sc_\alpha$ and so 
						\[\Vert T^{-1}(z)+e_2\Vert =\Vert T^{-1}(z)-e_2\Vert=2. \]
						This implies by Lemma~\ref{lemma6} that \[\Vert z+ f_2\Vert =\Vert z- f_2\Vert=2. \]
						So, there exist $H_1, H_2\in\Sc_\alpha$ such that
						\begin{align}\label{mmmm}\sum_{i\in H_1}\vert z(i)+f_2(i)\vert =2 \qquad\text{and}\qquad \sum_{i\in H_2}\vert z(i)-f_2(i)\vert =2 .\end{align}
						Now, we will show that 
						\begin{align}\label{mmm1} r:= \max\mathrm{supp}(f_k)\in H_1. \end{align}
						
						Assume first towards a contradiction that $m\notin H_1$. Then 
						\[\begin{aligned}\sum_{i\in H_1}\vert z(i)+f_2(i)\vert &= \sum_{i\in H_1}\biggl\vert \frac{1}{2}f_k(i)+f_2(i)\biggr\vert\\&\leqslant \frac{1}{2}\sum_{i\in H_1}\vert f_k(i)\vert + \sum_{i\in H_1}\vert f_2(i)\vert\leqslant \frac{3}{2},\end{aligned} \]
						which is a contradiction. 
						Hence $m\in H_1$. 
						
						Now, we observe that $H'_1:=H_1\cap \mathrm{supp}(f_k)$ is a $1$-set of $f_k$. Indeed, if $\sum_{i\in H_1}\vert f_k(i)\vert <1$, then since $m>\mathrm{supp}(f_k)\cup \mathrm{supp}(f_2)$, we have 
						\[\begin{aligned}\sum_{i\in H_1}\vert z(i)+f_2(i)\vert &=\sum_{i\in H_1\minus \{m\}}\vert z(i)+f_2(i)\vert + \vert z(m)\vert \\&\leqslant \sum_{i\in H_1\minus\{m\}}\vert z(i)\vert + \sum_{i\in H_1\minus\{m\}}\vert f_2(i)\vert+\frac{1}{2}\\&\leqslant \frac{1}{2}\sum_{i\in H_1\minus \{m\}}\vert f_k(i)\vert +\frac{3}{2}\\&=\frac{1}{2}\sum_{i\in H_1}\vert f_k(i)\vert +\frac{3}{2}<2,
						\end{aligned}\]
						which contradicts (\ref{mmmm}). Hence $H_1':=H_1\cap \mathrm{supp}(f_k)$ is a $1$-set of $f_k$. Since $ H_1\in \Sc_\alpha$ and $\Sc_\alpha$ is hereditary, it follows that $H'_1\cup\{m\}\in \Sc_\alpha$ and hence $H'_1\in \Sc_\alpha\minus \Sc_\alpha^{MAX}$.
						
						Finally,  assume towards a contradiction that $r:=\max\mathrm{supp}(f_k)\notin  H_1$. Since $H'_1\in \Sc_\alpha\minus \Sc_\alpha^{MAX}$ and $r>H'_1$, it follows by Fact~\ref{max} that 
						\[H'_1\cup \{r\}\in\Sc_\alpha,\] and so since $H'_1$ is a $1$-set of $f_k$, we get 
						\[\Vert f_k\Vert \geqslant\sum_{i\in H'_1\cup \{r\}}\vert f_k(i)\vert =\sum_{i\in H'_1}\vert f_k(i)\vert+\vert f_k(r)\vert >1, \]
						which is a contradiction. This proves (\ref{mmm1}). 
						
						In the same way, we show that 
						\begin{align}\label{mmm2}
							r\in H_2. 
						\end{align}
						Note that $f_2(r)\neq0$ because $\mathrm{supp}(f_k)\subseteq\mathrm{supp}(f_2)$ by (\ref{suppfk}). Hence, by Lemma~\ref{lemma1} and (\ref{mmmm}), (\ref{mmm1}) implies that 
						\[\mathrm{sgn}(f_2(r))=\mathrm{sgn}(z(r))=\mathrm{sgn}(f_k(r))\]
						and (\ref{mmm2}) implies that 
						\[\mathrm{sgn}(f_2(r))=-\mathrm{sgn}(z(r))=-\mathrm{sgn}(f_k(r))\]
						which is a contradiction.

						\bigskip
						
						\medskip
						
						So, in both cases, we get a contradiction, which means that when $2\in \mathrm{supp}(d_2)$, there is no $k\in \mathrm{supp}(d_2)$ such that $k\neq 2$, \textit{i.e.} $\mathrm{supp}(d_2)=\{2\}$. This completes the proof. 
					\end{proof}
					
					\begin{lemma}\label{iii}
						For any $k\in \N\minus \{1\}$, there exists $k'\in \N\minus \{1\}$ such that $f_k=\pm e_{k'}$.
					\end{lemma}
					\begin{proof}
						By Lemma~\ref{treslong}, we know that $d_2=\pm e_j$ for some $j\in \N\minus\{1\} $. If $k=j$, then $f_k=\pm e_2$  by Lemma~\ref{lemma6}. If $k\neq j$, then $k\notin \mathrm{supp}(d_2)$ and hence by Lemma~\ref{lemma11},  $2\notin \mathrm{supp}(f_k)$. So by Lemma~\ref{lemma15}, $f_k=\pm e_{k'}$ for some $k'\in \N\minus \{1\}$. 
					\end{proof}
					
					The next lemma is Proposition~\ref{main3} for $p=1$.
					
					\begin{lemma}\label{cle}
						For any $k\in \N$, there exists a sign $\theta_k$ such that $f_k=\theta_k e_k$ and $d_k=\theta_k e_k$.\end{lemma}
					\begin{proof}
						The proof is similar to that of Lemma~\ref{Fact6}.  
					\end{proof}
					
					% \begin{lemma}\label{cle}
						%   For any $k\in \N$, there exists a sign $\theta_k$ such that $f_k=\theta_k e_k$ (equivalently $d_k=\theta_k e_k$).    
						%  \end{lemma}
					\begin{lemma}\label{zero'}
						Let $x\in \Sp$ and $i\in \N$. If $x(i)=0$, then $T(x)(i)=0$ and  $T^{-1}(x)(i)=0$.
					\end{lemma}
					\begin{proof}
						The proof is the same as that of Lemma~\ref{zero}. We use Lemmas~\ref{lemma5} and~\ref{lemma6} instead of ~\ref{l5} and~\ref{l6} respectively.   
					\end{proof}
					
					\begin{lemma}\label{long}
						Let $x \in \Sp$ and $i\in \N$. If there exists a $1$-set $F$ of $x$ such that $i\in F$, then \[\vert T(x)(i)\vert =\vert T^{-1}(x)(i)\vert =\vert x(i)\vert.\]
					\end{lemma}
					\begin{proof}
						Let $i\in \N$ and $F$ be a $1$-set of $x$ such that $i\in F$. Then 
						\begin{align}\label{mmmm1}\sum_{k\in F\minus \{i\}}\vert x(k)\vert =1-\vert x(i)\vert .\end{align}
						We will distinguish two cases:
						
						\smallskip \textbf{Case 1:} $\vert x(i)\vert <\frac{1}{2}\cdot$
						
						\smallskip We will show first that \begin{align}\label{jjjj1}\vert T(x)(i)\vert \geqslant \vert x(i)\vert.\end{align}
						
						\smallskip
						Let $\theta$ be the sign of $x(i)$. We have by (\ref{mmmm1}) that 
						\[\begin{aligned}\Vert x+\theta e_i\Vert &\geqslant \sum_{k\in F}\vert x(k)+\theta e_i(k)\vert \\&=\sum_{k\in F\minus \{i\}}\vert x(k)\vert +\vert x(i)+\theta\vert \\&=1-\vert x(i)\vert +\vert x(i)\vert +1=2,
						\end{aligned}\]
						and 
						\[\begin{aligned}
							\Vert x-\theta e_i\Vert &\geqslant  \sum_{k\in F}\vert x(k)-\theta e_i(k)\vert \\&=\sum_{k\in F\minus \{i\}}\vert x(k)\vert +\vert x(i)-\theta\vert \\&=1-\vert x(i)\vert +1-\vert x(i)\vert \\&=2-2\vert x(i)\vert >1,
						\end{aligned}\]
						since $\vert x(i)\vert <\frac{1}{2}\cdot$
						
						Moreover, for any $F'\in \Sc_\alpha$, if $i\notin F'$, then 
						\[\sum_{k\in F'}\vert  x(k)-\theta e_i(k)\vert =\sum_{k\in F'}\vert  x(k)\vert \leqslant 1, \]
						and if $i\in F'$, then 
						\[\sum_{k\in F'\minus \{i\}}\vert x(k)\vert = \sum_{k\in F'}\vert x(k)\vert -\vert x(i)\vert \leqslant 1-\vert x(i)\vert \] and so 
						\[\sum_{k\in F'}\vert  x(k)-\theta e_i(k)\vert=\sum_{k\in F'\minus \{i\}}\vert x(k)\vert+\vert x(i)- \theta\vert \leqslant 1-\vert x(i)\vert +1-\vert x(i)\vert =2-2\vert x(i)\vert.  \]
						Therefore, 
						\[\Vert  x+\theta e_i\Vert =2 \qquad\text{and}\qquad \Vert  x-\theta e_i\Vert=2-2\vert x(i)\vert >1.\]
						Hence, by Lemma~\ref{lemma6} and since $f_i=\theta_i e_i$ (by Lemma~\ref{cle}), it follows that 
						\begin{align}\label{jjjj}\Vert T(x)+\theta \theta_i e_i\Vert =2\qquad \text{and}\qquad \Vert T(x)-\theta \theta_i e_i\Vert =2-2\vert x(i)\vert >1.  \end{align}
						So, by Lemma~\ref{lemma1}, there exists $G\in \mathcal{A}_{T(x)}^1$ such that $i\in G$,
						\[\sum_{k\in G}\vert T(x)(k)+\theta\theta_i e_i(k)\vert =2, \]
						and $T(x)(i)=0$ or $\mathrm{sgn}(T(x)(i))=\theta \theta_i$. 
						
						Since $x(i)\neq 0$ (because $i\in F$ and $F$ is a $1$-set of $x$), it follows by Lemma~\ref{zero'} that $T(x)(i)\neq 0$. Hence, $\mathrm{sgn}(T(x)(i))=\theta \theta_i$.
						%If $T(x)(i)=0$, then $\sum_{k\in G}\vert T(x)(k)-\theta\theta_i e_i(k)\vert =2$, which is a contradiction because $\Vert T(x)-\theta \theta_i e_i\Vert =2-2\vert x(i)\vert <2$. (Note that $x(i)\neq 0$ because $i\in F$ where $F$ is a $1$-set of $x$.) 
						
						Since $G\in \mathcal{A}_{T(x)}^1$, it follows that 
						\[\sum_{k\in G\minus \{i\}}\vert T(x)(k)\vert =1-\vert  T(x)(i)\vert, \]
						and so 
						\[\begin{aligned}
							\sum_{k\in G}\vert T(x)(k)-\theta\theta_i e_i(k)\vert &=\sum_{k\in G\minus \{i\}}\vert T(x)(k)\vert+\vert T(x)(i)-\theta\theta_i\vert \\&=1-\vert  T(x)(i)\vert+1-\vert T(x)(i)\vert\\& =2-2\vert T(x)(i)\vert.
						\end{aligned}\]
						Hence, 
						\[2-2\vert x(i)\vert =\Vert T(x)-\theta\theta_ie_i\Vert \geqslant  \sum_{k\in G}\vert T(x)(k)-\theta\theta_i e_i(k)\vert=2-2\vert T(x)(i)\vert. \]
						This proves (\ref{jjjj1}).

						Now, we will show that $\vert T(x)(i)\vert <\frac{1}{2}\cdot$ Towards a contradiction, assume that $\vert T(x)(i)\vert \geqslant\frac{1}{2}\cdot$ Let $G'\in \Sc_\alpha$. If $i\notin G'$, then 
						\[\sum_{i\in G'}\vert  T(x)(k)-\theta\theta_i e_i(k)\vert=\sum_{i\in G'}\vert  T(x)(k)\vert \leqslant 1.\]
						If $i\in G'$, then 
						\[\sum_{k\in G'\minus \{i\}}\vert T(x)(k)\vert =\sum_{k\in G'}\vert T(x)(k)\vert-\vert T(x)(i)\vert \leqslant 1-\vert T(x)(i)\vert, \]
						and since $\mathrm{sgn}(T(x)(i))=\theta \theta_i$, it follows that  
						\[\begin{aligned}\sum_{i\in G'}\vert  T(x)(k)-\theta\theta_i e_i(k)\vert&=\sum_{k\in G'\minus \{i\}}\vert T(x)(k)\vert+\vert T(x)(i)-\theta \theta_i\vert \\&\leqslant 1-\vert T(x)(i)\vert+1-\vert T(x)(i)\vert\\&=2-2\vert T(x)(i)\vert \leqslant1,\end{aligned}\]
						since we are assuming that $\vert T(x)(i)\vert \geqslant\frac{1}{2}\cdot$ Therefore, 
						$\Vert T(x)-\theta \theta_i e_i\Vert \leqslant1, $
						which contradicts (\ref{jjjj}). Hence, $\vert T(x)(i)\vert <\frac{1}{2}\cdot$
						
						So, we have $\vert T(x)(i)\vert <\frac{1}{2}$ and $G\cap \mathrm{supp}(T(x))$ is a $1$-set of $T(x)$ that contains $i$ (because $G\in \mathcal{A}_{T(x)}^1$ and $T(x)(i)\neq 0$). Hence, if we repeat the same proof for $T^{-1}$ and $T(x)$ instead of $T$ and $x$, we obtain that  
						\[\vert T^{-1}(T(x))(i)\vert \geqslant \vert T(x)(i)\vert,  \]
						\textit{i.e.}
						\begin{align}\label{jjjj2}
							\vert x(i)\vert \geqslant  \vert T(x)(i)\vert. 
						\end{align}
						Hence, (\ref{jjjj1}) and (\ref{jjjj2}) imply that $\vert T(x)(i)\vert=\vert x(i)\vert$. Similarly, we show that $\vert T^{-1}(x)(i)\vert=\vert x(i)\vert$. 
						
						\medskip
						\textbf{Case 2:} $\vert x(i)\vert \geqslant\frac{1}{2}\cdot$
						
						\smallskip Let us  distinguish the following two subcases: 
						
						\smallskip \textbf{Case 2.1:} $F=\{i,j\}$ where $j\neq i\in \N\minus \{1\}$ and $\vert x(i)\vert=\vert x(j)\vert=\frac{1}{2}\cdot$
						
						\smallskip
						Let $\theta$ be the sign of $x(i)$.  Since $F=\{i,j\}\in \Sc_\alpha$, it follows that 
						\[\Vert x+\theta e_i\Vert \geqslant \vert x(i)+\theta \vert +\vert x(j)\vert =1+\vert x(i)\vert  +\vert x(j)\vert=2. \]
						%and 
						%\[\Vert x-\theta e_i\Vert \geqslant \vert x(i)-\theta \vert +\vert x(j)\vert =1-\vert x(i)\vert  +\vert x(j)\vert=1.  \]
						%Moreover, for any $F'\in \Sc_\alpha$, if $i\notin F'$, then 
						%\[\sum_{k\in F'}\vert x(k)-\theta e_i(k)\vert =\sum_{k\in F'}\vert x(k)\vert \leqslant 1,\]
						%and, if $i\in F'$ then 
						%\[\sum_{k\in F'}\vert x(k)-\theta e_i(k)\vert =\sum_{k\in F'\minus\{i\}}\vert x(k)\vert+\vert x(i)-\theta\vert \leqslant 1-\vert x(i)\vert +1-\vert x(i)\vert =1. \]
						Hence, $\Vert x+\theta e_i\Vert=2$ and so by Lemma~\ref{lemma6} and Lemma~\ref{cle}, 
						%\[\Vert x+\theta e_i\Vert=2\qquad\text{and}\qquad \Vert x-\theta e_i\Vert=1.  \]
						%This implies that 
						\[\Vert T(x)+\theta\theta_ie_i\Vert =2. \]
						So by Lemma~\ref{lemma1}, there exists $G\in \mathcal{A}_{T(x)}^1$ such that $i\in G$ and 
						\[\sum_{k\in G}\vert T(x)(k)+\theta\theta_ie_i(k)\vert =2. \]
						%and $T(x)(i)=0$ or $\mathrm{sgn}(T(x)(i))=\theta\theta_i$. 
						
						Since $x(i)\neq 0$, it follows by Lemma~\ref{zero'} that $T(x)(i)\neq 0$. Hence $G\cap \mathrm{supp}(T(x))$ is a $1$-set of $T(x)$ that contains $i$. This implies that  if $\vert T(x)(i)\vert <\frac{1}{2}$, then by Case 1, \[\vert T(x)(i)\vert =\vert T^{-1}(T(x))(i)\vert =\vert x(i)\vert =\frac{1}{2},\]
						which is a contradiction. Hence $\vert T(x)(i)\vert \geqslant \frac{1}{2}\cdot$ Similarly, we show that $\vert T(x)(j)\vert \geqslant \frac{1}{2}\cdot$
						
						Since $\{i,j\}\in \Sc_\alpha$ and $\Vert T(x)\Vert =1$, it follows that \[\vert T(x)(i)\vert =\vert T(x)(j)\vert=\frac{1}{2}\cdot\]
						
						In the same way, we show that $\vert T^{-1}(x)(i)\vert =\frac{1}{2}\cdot$
						
						\smallskip
						\textbf{Case 2.2:} $\vert x(i)\vert \geqslant\frac{1}{2}$ and $\vert x(k)\vert<\frac{1}{2}$ for any $k\in F\minus \{i\}$.
						
						\smallskip   
						As in the previous case, it is easy to show that \[\Vert T(x)+\theta\theta_ie_i\Vert =2,\]
						where $\theta$ is the sign of $x(i)$. Hence, by Lemma~\ref{lemma1}, there exists $G\in \mathcal{A}_{T(x)}$ such that $i\in G$ and \[\sum_{k\in G}\vert T(x)(k)+\theta\theta_ie_i(k)\vert =2. \]
						Moreover, $T(x)(i)\neq0$ by Lemma~\ref{zero'}. 
						Now, we will show that 
						\begin{align}\label{llll}
							\frac{1}{2}\leqslant \vert T(x)(i)\vert \leqslant\vert x(i)\vert.  
						\end{align}
						Towards a contradiction, assume that $\vert T(x)(i)\vert <\frac{1}{2}\cdot$ Since $G':=G\cap \mathrm{supp}(T(x))$ is a $1$-set of $T(x)$ that contains $i$, it follows by Case 1 that \[\vert x(i)\vert =\vert T^{-1}(T(x))(i)\vert =\vert T(x)(i)\vert <\frac{1}{2},\]
						which is a contradiction since $\vert x(i)\vert \geqslant\frac{1}{2}\cdot$ Hence $\vert T(x)(i)\vert \geqslant \frac{1}{2}\cdot$
						
						For any $k\in F\minus\{i\}$, we have $\vert x(k)\vert <\frac{1}{2}$, and hence by Case 1, $\vert T(x)(k)\vert =\vert x(k)\vert $. Assume towards a contradiction that $\vert T(x)(i)\vert >\vert x(i)\vert$. Since $F$ is a $1$-set of $x$, it follows that 
						\[ \begin{aligned}
							\Vert T(x)\Vert &\geqslant  \sum_{k\in F}\vert T(x)(k)\vert \\&=\vert T(x)(i)\vert + \sum_{k\in F\minus \{i\}}\vert T(x)(k)\vert \\&>\vert x(i)\vert + \sum_{k\in F\minus \{i\}}\vert T(x)(k)\vert=1,
						\end{aligned}\]
						which is a contradiction. So, we have shown (\ref{llll}).
						
						Note that if $\vert x(i)\vert =\frac{1}{2}$, then (\ref{llll}) implies that $\vert T(x)(i)\vert =\frac{1}{2}\cdot$ 
						
						So, we will now suppose that $\vert x(i)\vert >\frac{1}{2}\cdot$
						
						If $G':=G\cap \mathrm{supp}(T(x)) =\{i,j\}$ where $j\neq i$ and $\vert T(x)(j)\vert =\frac{1}{2}$, then $\vert T(x)(i)\vert =\frac{1}{2}$ because $G'$ is a $1$-set of $T(x)$. Hence by Case 2.1, we have 
						\[\vert x(i)\vert =\vert T^{-1}(T(x))(i)\vert=\vert T(x)(i)\vert =\frac{1}{2},\]
						which is a contradiction since $\vert x(i)\vert >\frac{1}{2}\cdot$ Hence, for any $k\in G'\minus \{i\}$, we have $\vert T(x)(k)\vert <\frac{1}{2}\cdot$ So if we repeat the same argument (for $T^{-1}$ and $T(x)$ instead of $T$ and $x$) we obtain 
						\[\vert T^{-1}(T(x))(i)\vert \leqslant \vert T(x)(i)\vert, \]
						\textit{i.e.}
						\begin{align}\label{llll1}
							\vert x(i)\vert  \leqslant \vert T(x)(i)\vert.
						\end{align}
						Hence, (\ref{llll}) and (\ref{llll1}) imply that $\vert T(x)(i)\vert=\vert x(i)\vert$. Similarly, we show that $\vert T^{-1}(x)(i)\vert=\vert x(i)\vert$.
					\end{proof}
					
					\begin{lemma}\label{final'}
						For any $x\in \Sp$ and $i\in \N$, we have $T(x)(i)=\theta_i x(i)$. 
					\end{lemma}
					\begin{proof}
						The proof is the same as that of Lemma~\ref{final}. We use Lemmas~\ref{lemma3},~\ref{lemma6},~\ref{zero'} and~\ref{long} instead of~\ref{l2},~\ref{l6},~\ref{zero} and~\ref{imp}, respectively. 
					\end{proof}
					This completes the proof of Theorem~\ref{main2} and consequently Theorem~\ref{main} for $p=1$.
					
				\end{document}